\documentclass[11pt]{article}

\usepackage{stmaryrd}
\usepackage[utf8]{inputenc}
\usepackage[vietnamese,english]{babel}
\usepackage[all, cmtip]{xy}

\usepackage{amsmath}
\usepackage{amsxtra}
\usepackage{amscd}
\usepackage{amsthm}
\usepackage{amsfonts}
\usepackage{amssymb}
\usepackage{bbm}

\usepackage[bitstream-charter]{mathdesign}
\usepackage[T1]{fontenc}
\usepackage{fancyhdr}
\usepackage{tikz}
\usepackage{tikz-cd}
\usetikzlibrary{matrix,calc,arrows}

\numberwithin{equation}{section}

\newtheorem{theorem}{Theorem}
\newtheorem{lemma}[theorem]{Lemma}
\newtheorem{lemma-construction}[theorem]{Lemma-Construction}
\newtheorem{proposition}[theorem]{Proposition}
\newtheorem{definition}[theorem]{Definition}
\newtheorem{corollary}[theorem]{Corollary}
\newtheorem{conjecture}[theorem]{Conjecture}
\numberwithin{theorem}{section}
\theoremstyle{definition}

\newtheorem{thm}{Theorem}[section]
\newtheorem{prop}[thm]{Proposition}

\theoremstyle{remark}
\newtheorem{remark}{Remark}[section]
\newtheorem{example}[remark]{Example}

\newcommand\indlim\varinjlim

\newcommand\cA{\mathcal{A}}
\newcommand\cB{\mathcal{B}}

\newcommand\cE{\mathcal{E}}
\newcommand\cF{\mathcal{F}}

\newcommand\cL{\mathcal{L}}
\newcommand\cM{\mathcal{M}}

\newcommand\cO{\mathcal{O}}
\newcommand\cP{\mathcal{P}}

\newcommand\cT{\mathcal{T}}

\newcommand\cV{\mathcal{V}}

\let\fg\undefined

\newcommand{\fc}{{\mathfrak c}}

\newcommand{\fg}{{\mathfrak g}}

\newcommand{\fm}{{\mathfrak m}}

\newcommand{\ft}{{\mathfrak t}}

\newcommand{\fC}{{\mathfrak C}}

\newcommand{\fS}{{\mathfrak S}}

\newcommand{\frakg}{{\mathfrak g}}

\newcommand\bbA{\mathbb{A}}

\newcommand\CC{\mathbb{C}}
\newcommand\DD{\mathbb{D}}

\newcommand\GG{\mathbb{G}}

\newcommand\ZZ{\mathbb{Z}}

\newcommand\bA{{\bbA}}
\newcommand\bC{{\CC}}

\newcommand\bG{{\GG}}

\newcommand\bZ{{\ZZ}}

\newcommand\Z{\mathbb{Z}}

\newcommand\A{\mathbb{A}}
\newcommand\N{\mathbb{N}}

\def\d{\mathrm{d}}

\newcommand\rH{\mathrm{H}}

\newcommand\rT{\mathrm{T}}

\newcommand\gl{\mathfrak{gl}}
\newcommand\GL{\mathrm{GL}}

\newcommand\SL{\mathrm{SL}}

\newcommand\id{\mathrm{id}}
\newcommand\Spec{\mathrm{Spec}}

\newcommand\Gm{\mathbb{G}_m}

\newcommand\Hom{\mathrm{Hom}}
\newcommand\Ext{\mathrm{Ext}}

\newcommand\End{\mathrm{End}}

\newcommand\Pic{\mathrm{Pic}}
\newcommand\Sym{\mathrm{S}}
\newcommand\Lie{\mathrm{Lie}}

\newcommand\ad{\mathrm{ad}}

\newcommand\rss{\mathrm{rss}}

\newcommand\Supp{\mathrm{Supp}}

\newcommand\Chow{\mathrm{Chow}}
\newcommand\Hilb{\mathrm{Hilb}}
\newcommand\HC{\mathrm{HC}}

\newcommand\codim{\mathrm{codim}}

\newcommand\length{\mathrm{length}}

\newcommand{\rW}{{\mathrm W}}

\newcommand{\mO}{\mathcal{O}}

\newcommand{\sB}{\mathscr{B}}
\newcommand{\sP}{\mathscr{P}}

\newcommand{\on}{\operatorname}

\newcommand{\is}{\simeq}

\newcommand{\tC}{\widetilde C}

\newcommand{\quash}[1]{}  %%Anything in \quash is ignored
\newcommand{\nc}{\newcommand}

\newcommand\sd{\on{sd}}

\nc{\al}{{\alpha}} \nc{\be}{{\beta}} \nc{\ga}{{\gamma}}
\nc{\ve}{{\varepsilon}} %\nc{\Ga}{{\Gamma}} %\nc{\la}{{\lambda}}
\nc{\La}{{\Lambda}}

\newcommand{\beqn}{\begin{equation*}}
\newcommand{\eeqn}{\end{equation*}}

\newcommand{\beq}{\begin{equation}}
\newcommand{\eeq}{\end{equation}}

\newcommand{\pol}{\on{pol}}

\title{On the Hitchin morphism for higher dimensional varieties}
\author{T.H. Chen, B.C. Ngô}

\date{}
\begin{document}
\maketitle

\begin{abstract}
In this paper, we explore the structure of the Hitchin morphism for higher dimensional varieties. We show that the Hitchin morphism factors through a closed subscheme of the Hitchin base, which is in general a non-linear subspace of lower dimension.
We conjecture that the resulting morphism, which we call the spectral data morphism, is surjective. 
In the course of the proof, we establish connections between the Hitchin morphism for higher dimensional varieties, the invariant theory of the commuting schemes, and Weyl's polarization theorem. 
We use the factorization of the Hitchin morphism to construct the spectral and cameral covers.
In the case of general linear groups and algebraic surfaces, we show that spectral surfaces admit canonical finite Cohen-Macaulayfications, which we call the Cohen-Macaulay spectral surfaces, and we use them to obtain a description of the generic fibers of the Hitchin morphism similar to the case of curves.
Finally, we study the Hitchin morphism for some classes of algebraic surfaces.
\end{abstract}

\section{Introduction}
For a smooth projective curve $X$ over a field $k$, and a split reductive group $G$ over $k$ of rank $n$, a $G$-Higgs bundle over $X$ is a pair $(E,\theta)$ consisting of a principal $G$-bundle $E$ over $X$ and an element 
$\theta\in \rH^0(X,\ad(E)\otimes_{\mO_X}\Omega^1_{X})$ called a Higgs field, 
where $\ad(E)$ is the adjoint vector bundle associated with $E$ and $\Omega^1_{X}$ is the sheaf of $1$-forms of $X$. In \cite{Hitchin}, Hitchin constructed a completely integrable system on the moduli space $\cM_X$ of $G$-Higgs bundles over a curve $X$. This system can be presented as a morphism 
$h_X:\cM_X\to\cA_X$ 
where $\cM_X$ is the moduli space of Higgs bundles and $\cA_X$ is the affine space 
\begin{equation} \label{A_X}
	\cA_X=\bigoplus_{i=1}^n \rH^0(X,\Sym^{e_i}\Omega^1_{X})
\end{equation}
where $\Sym^{e_i}\Omega^1_{X}$ is the $e_i$-th symmetric power of $\Omega^1_{X}$. The morphism $h_X$ is known as the Hitchin fibration. For curves $X$ of genus $g_X>1$, $h_X$ is surjective and its generic fiber is isomorphic to a disjoint union of abelian varieties if we discard automorphisms. This work aims at addressing these basic properties of the Hitchin morphism for higher dimensional algebraic varieties. 

Over a higher dimensional algebraic variety $X$, a $G$-Higgs bundle is a $G$-bundle $E$ over $X$ equipped with a Higgs field
\begin{equation} \label{G-Higgs-field}
\theta\in \rH^0(X,\ad(E) \otimes_{\mO_X} \Omega^1_{X})
\end{equation}
where $\ad(E)$ is the adjoint vector bundle of $E$,
satisfying the integrability condition $\theta\wedge \theta=0$. 
With given local coordinates $z_1,\ldots, z_d$ in a neighborhood $U$ of $x\in X$ and given local trivialization of $E$, we can write $\theta=\sum_{i=1}^d \theta_i \d z_i$
where $\theta_i : U\to \fg$ are functions on $U$ with values in the Lie algebra $\fg$ of $G$. The integrability condition satisfied by the Higgs field is 
$$[\theta_i,\theta_j]=0$$ 
for all $1\leq i,j \leq d$. Hitchin's construction, generalized to higher dimensional varieties by Simpson
\cite{Simpson 2}, provides a morphism $h_X:\cM_X\to \cA_X$ where $\cA_X$ is the affine space \eqref{A_X}. 

For general higher dimensional algebraic varieties, the Hitchin morphism  is very far from being surjective. We note that $h_X(E,\theta)$ could be defined for any $\theta\in \rH^0(X,\ad(E) \otimes_{\mO_X} \Omega^1_{X})$ no matter if it satisfies the integrability condition $\theta\wedge \theta =0$ or not. We aim at understanding the equations on $\cA_X$ implied by the integrability condition $\theta\wedge \theta=0$.

Our study of the Hitchin morphism for higher dimensional varieties follows the method of \cite{Ngo} in the one dimensional case. Namely, instead of studying the Hitchin morphism for a given variety $X$, we study certain universal morphisms independent of $X$. Those morphisms have to do with the construction of $G$-invariant functions on the scheme $\fC^d_G$ of commuting elements $x_1,\ldots,x_d$ in the Lie algebra $\fg=\Lie(G)$. The reductive group $G$ acts diagonally on $\fC^d_G$ by the adjoint action on $x_1,\ldots,x_d$. 

Our study of $G$-invariant functions on $\fC^d_G$ can roughly be divided into two parts. First, we investigate the generalization of the Chevalley restriction theorem to the commuting scheme. Second, we investigate the subring of $G$-invariant functions on $\fC^d_G$ derived from Weyl's polarization method. Both of these investigations are hindered by notoriously difficult problems in commutative algebra, for instance the question whether the categorical quotient $\fC^d_G\sslash G$ is reduced. We are able to prove the reducedness of $\fC^d_G\sslash G$ in the case $G=\GL_n$ generalizing a theorem of Gan-Ginzburg in the case $d=2$ \cite{GG}. Although we can't prove the reducedness of $\fC^d_G\sslash G$ for general reductive groups, we can work around it and address the problem of description of the image of the Hitchin morphism. Moreover, we will state and hierarchize certain problems which are related to 
%and weaker than 
the reducedness of $\fC^d_G\sslash G$, which seem to be worthy of further investigation.

Here is a summary of our results. For a higher dimensional proper smooth algebraic variety $X$, the Hitchin morphism $h_X:\cM_X\to \cA_X$ where $\cM_X$ is the moduli stack of Higgs bundles on $X$ and $\cA_X$ is the affine space defined by the formula \eqref{A_X}, is not surjective in general. We will define a closed subscheme $\cB_X$ of $\cA_X$, which is in general a non linear subspace of much lower dimension and prove that $h_X$ factors through $\cB_X$ (or rather, a thickening of $\cB_X$, see Section \ref{postulated image}). 
We conjecture that the resulting morphism $\cM_X\to \cB_X$, which we call the spectral data morphism, is surjective. 
In the course of the proof, we establish the connections between the
Hitchin morphisms for higher dimensional varieties, the invariant theory 
of the commuting schemes, and Weyl's polarization theorem in classical invariant theory.

We use the factorization of the Hitchin morphism to construct spectral and cameral covers and establish basic properties of them. In particular, we will see that, unlike the case of curves, the spectral and cameral covers are generally not flat in higher dimension.
In the case $G=\GL_n$ and $\dim(X)=2$, we construct an open subset $\cB^\heartsuit_X$ of $\cB_X$ such that for every $b\in \cB^\heartsuit_X$, 
the corresponding spectral surface admits a canonical 
finite Cohen-Macaulayfication, called the Cohen-Macaulay spectral surface,
and we use it to 
obtain a description of the Hitchin fiber $h_X^{-1}(b)$ similar to the case of curves. 
In particular, 
we show that $h_X^{-1}(b)$
is non empty for $b\in \cB^\heartsuit_X$
and there is a natural action of 
the
Picard stack $\mathscr P_b$ of line bundles on 
the Cohen-Macaulay spectral surface on $h_X^{-1}(b)$. 
We also construct an open subset $\cB_X^\diamondsuit$ of $\cB_X^\heartsuit$ such that for all $b\in \cB_X^\diamondsuit$ the fiber $h_X^{-1}(b)$ is isomorphic to a disjoint union of abelian varieties after we discard automorphisms. For some class of algebraic surfaces (including elliptic surfaces), we can prove that $\cB_X^\diamondsuit$ is an open dense subset of $\cB_X^\heartsuit$, which is an open dense subset of $\cB_X$. 

Throughout this paper, we fix an algebraically closed field $k$ of characteristic zero. To remove or weaken the assumption on the characteristic of $k$, we would have to refine many deep results in invariant theory. We will come back to deal with this task in a future work.

\section{Characteristics of Higgs bundles over curves}

Hitchin's construction was revisited in \cite{Ngo} from the point of view of the theory of algebraic stacks. In loc. cit. the Hitchin morphism $h_X:\cM_X\to\cA_X$ was derived from a natural morphism of algebraic stacks 
\begin{equation}
	h: [\fg/G] \to \fg\sslash G
\end{equation} 
where $[\fg/G]$ and $\fg\sslash G$ are the quotients of the Lie algebra of $\fg$ by the adjoint action of $G$ in the framework of algebraic stacks and
geometric invariant theory respectively. We recall that for every test scheme $S$, the groupoid of $S$-points of $[\frakg/G]$ consist of all pairs $(E,\theta)$ where $E$ is a principal $G$-bundles over $S$ and $\theta\in\rH^0(S,\ad(E))$ is a global section of the adjoint vector bundle $\ad(E)$ obtained from $E$ by pushing out by the adjoint representation $\ad:G\to\GL(\fg)$ of $G$. The categorical quotient $\fg\sslash G$ is the affine scheme $ \fg\sslash G=\Spec(k[\fg]^G)$ where $k[\fg]^G$ is the ring of $G$-invariant functions on $\fg$. The concept of categorical quotient $\fg\sslash G$ was devised by Mumford in \cite{Mumford} by which he means the initial object in the category of pairs $(q,Q)$ where $Q$ is a $k$-scheme and $q:\fg\to Q$ is a $G$-invariant morphism. 

We will also use the Chevalley restriction theorem.  Let us denote by $\ft$ a Cartan algebra, and $W$ its Weyl group. Since $W$-conjugate elements in $\ft$ are $G$-conjugate as elements of $\fg$, the restriction of a $G$-invariant function on $\fg$ to $\ft$ is $W$-invariant and therefore defines a homomorphism of algebras $k[\fg]^G\to k[\ft]^W$. The Chevalley restriction theorem asserts that this map is an isomorphism. This is equivalent to stating that
the morphism between the categorical quotients
\begin{equation}\label{Chevalley}
	\ft\sslash W \to \fg\sslash G,
\end{equation}
is an isomorphism.

Let us denote $\fc=\ft\sslash W$. Since $W$ acts on $\ft$ as a reflection group, after another theorem of Chevalley, $\fc$ is also isomorphic to an affine space. The scalar action of $\Gm$ on $\ft$ induces an action of $\Gm$ on $\fc$. In fact, we can choose coordinates $c_1,\ldots,c_n$ of the affine space $\fc$ that are homogenous as polynomial functions of $\ft$, that is,
\begin{equation}\label{Gm-action}
	t(c_1,\ldots,c_n)=(t^{e_1} c_1,\ldots,t^{e_n}c_n).
\end{equation}
The integers $e_1,\ldots,e_n$ are independent of the choice of $c_1,\ldots,c_n$.

Before proceeding further with the construction of the Hitchin morphism for curves, and as preparation for the higher dimensional case, let us state an elementary yet useful fact. Let $V$ be a finite dimensional $k$-vector space. 
The space of morphisms $f:V\to \bA^1$ satisfying $f(tv)=t^e f(v)$ can be canonically identified with the $e$-th symmetric power $\Sym^e V^*$ of the dual vector space $V^*$. This is equivalent to saying that the scalar action of $\Gm$ on $V$ gives rise to the graduation of the algebra of polynomial functions on $V$, i.e., $\Sym (V^*)=\bigoplus_{e\in\Z_{\geq 0}} \Sym^e V^*$. Although it may seem completely obvious, this is a useful fact that shouldn't be overlooked. For instance, for $e=1$, this says that any $\Gm$-equivariant polynomial map $f:V\to \A^1$, i.e., a polynomial map satisfying $f(tv)=t f(v)$, is automatically linear. For $d=2$, all polynomial maps $f:V\to \A^1$ satisfying $f(tv)=t^2f(v)$ is automatically quadratic and so on.  

A Higgs field $\theta\in\rH^0(X,\ad(E)\otimes_{\mO_X} \Omega^1_X)$ can be seen as an $\cO_X$-linear map $\cT_X\to \ad(E)$, where $\cT_X$ is the $\cO_X$-module of local sections of the tangent bundle $T_X$ of $X$, satisfying the integrability condition \eqref{integrability}. As the integrability condition is void when $X$ is a smooth algebraic curve, it will be ignored in this section. We note that a $\cO_X$-linear map $\cT_X\to \ad(E)$ is the same as a $\Gm$-equivariant morphism $\theta:T_X\to [\fg/G]$ lying over the map $X\to \mathbb BG$ corresponding to the $G$-bundle $E$. Here $\mathbb BG$ is the classifying stack of $G$.
By composing with the morphism $[\fg/G]\to \fg\sslash G$ and the inverse of the isomorphism $\fc\to \fg\sslash G$, we get a $\Gm$-equivariant morphism $a:T_X\to \fc$. For $i=1,\ldots,n$, by composing with the functions $c_i:\fc\to \A^1$, we obtain $\Gm$-equivariant morphisms $a_i:T_X\to \bA^1_{e_i}$ where $\bA^1_{e_i}$ is a copy of the affine line on which  $\Gm$ acts by the formula $t.x=t^{e_i}x$. 
We note that the space of all $\Gm$-equivariant functions $a_i:T_X\to \bA^1_{e_i}$ is the affine space of global section of the $e_i$-th symmetric power of the cotangent bundle $T_X^*$ of $X$. Finally, we obtain the Hitchin morphism $h_X:\cM_X\to \cA_X$ where $\cA_X$ is the affine space \eqref{A_X}.

The main result of \cite{Hitchin} asserts that, under the assumption $g_X \geq 2$, the generic fiber is isomorphic to a union of abelian varieties if we ignore isotropy groups. For instance, in the case $G=\GL_n$, Hitchin defines for every $a \in \cA_X$ a spectral curve $X^\bullet_a$. As $a$ varies, the spectral curves $X^\bullet_a$ form a linear system on the cotangent bundles of $X$. The assumption on the genus $g_X\geq 2$ implies that the linear system is ample and its generic member is a smooth projective curve. If $X^\bullet_a$ is smooth, the Hitchin fiber $\cM_a=h_X^{-1}(a)$ is isomorphic to the Picard stack $\Pic(X^\bullet_a)$ which is isomorphic to a disjoint union of abelian varieties if we ignore automorphisms. For classical groups, Hitchin also constructs certain spectral curves using their standard representations. For a general reductive group, Donagi constructs a cameral cover $\tilde X_a$ of $X$ for every $a\in \cA_X$ and proves that the Hitchin fiber $\cM_a$ is isomorphic to a union of abelian varieties if the cameral cover $\tilde X_a$ is a smooth curve. 

Since we will attempt to generalize the construction of cameral curves for Higgs bundles over higher dimensional varieties, let us recall their construction in the case of curves. The construction, due to Donagi \cite{Donagi}, derives the cameral covering $\pi_a:\tilde X_a\to X$ from the Cartesian diagram
\begin{equation}
\begin{tikzcd}
\tilde X_a \arrow{r}{\tilde a} \arrow{d}[swap]{\pi_a}
& {[\ft/\Gm]} \arrow{d}{} \\
X \arrow{r}[swap]{a}
& {[\fc/\Gm]}
\end{tikzcd}
\end{equation}
where the morphism $a:X\to [\fc/\Gm]$ at the bottom line comes from the $\Gm$-equivariant morphism $a:T_X\to \fc$. Since the morphism $\pi:\ft\to \fc$ is finite and flat, $\pi_a$ also has these properties. Away from the discriminant locus $\on{discr}_G\subset\fc$, the morphism $\pi:\ft\to\fc$ is finite étale and Galois with Galois group $W$.  In \cite{Ngo}, we denote by $\cA_X^\heartsuit$ the open subset of $\cA_X$ consisting of maps $a:X\to [\fc/\Gm]$ whose image is not contained in $[\on{discr}_G/\bG_m]$. By construction, for $a\in\cA_X^\heartsuit$, $\tilde X_a\to X$ is generically a finite étale Galois morphism with Galois group $W$. The fibers $\cM_a$ are much better understood under the assumption $a\in \cA_X^\heartsuit$. In particular, there is a natural Picard groupoid $\cP_a$, constructed in \cite{Ngo}, acting on $\cM_a$ with a dense open orbit.

\section{The Higgs stack and the universal spectral data morphism}

Let $X$ be a proper smooth variety of dimension $d$ over $k$. A $G$-Higgs bundle over $X$ is a $G$-bundle $E$ over $X$ equipped with a $\cO_X$-linear map $\theta:\cT_X \to \ad(E)$ from the tangent sheaf $\cT_X$ of $X$ to the adjoint vector bundle $\ad(E)$ of $E$ satisfying the integrability condition: for all local sections $v_1,v_2$ of $\cT_X$ we have 
\begin{equation}\label{integrability}
	[\theta(v_1),\theta(v_2)]=0.
\end{equation}

Let $\frak C_G^d\subset\fg^d$ be the commuting scheme.
It is defined as the scheme-theoretic
zero fiber of the commutator map
\[\fg^d\to\prod_{i<j} \fg,\ \ \  (\theta_1,...,\theta_d)\to\prod_{i<j} [\theta_i,\theta_j].\] 
The   
$k$-points of $\frak C_G^d$ consist of $(\theta_1,...,\theta_d)\in\fg^d(k)$ such that 
$[\theta_i,\theta_j]=0$ for $1\leq i,j\leq d$.
%Let $\fC^d_G$ be the closed subscheme of $\fg^d$ consisting of $(\theta_1,\ldots,\theta_d)\in \fg^d$ such that $[\theta_i,\theta_j]=0$ for all indices $i,j$ with $1\leq i,j \leq d$. 
We note that the commuting relations are automatically satisfied in the case $d=1$. Let $V_d$ denote the dual vector space of 
$k^d$ equipped with the standard basis $v_1,\ldots,v_d$. We will identify $\fg^d$ with the space of all linear maps $\theta:V_d\to \fg$ by attaching to $(\theta_1,\ldots,\theta_d)\in \fg^d$ the unique linear map $\theta:V_d\to \fg$ satisfying $\theta(v_i)=\theta_i$. The commuting scheme $\fC^d_G$ can then be identified with the closed subscheme of $\fg^d$ consisting of all $k$-linear maps $\theta:V_d\to \fg$ such that $[\theta(v),\theta(v')]=0$ for all $v,v'\in V_d$. 

Granted with this description of $\fC^d_G$, we have an action of $\GL_d\times G$ on $\fC^d_G$ coming from the natural action of $\GL_d$ on $V_d$ and the adjoint action of $G$ on $\fg$. We will call the quotient 
\begin{equation}
	[\fC^d_G/(\GL_d\times G)]
\end{equation}
in the sense of algebraic stack, the 
Higgs stack. It attaches to every test scheme $S$ the groupoid of triples $(\cV,\cE,\theta)$ consisting of a vector bundle $\cV$ of rank $d$ over $S$, a principal $G$-bundle $\cE$ over $S$, and a $\cO_S$-linear map $\theta:\cV\to \ad(\cE)$ satisfying $[\theta(v),\theta(v')]=0$ for all local sections $v,v'$ of $\cV$. A Higgs field on a $d$-dimensional proper smooth variety $X$ can be represented by a map
\begin{equation}\label{theta-local}
	\theta: X\to [\fC^d_G/(\GL_d\times G)]
\end{equation}
lying over the map $X\to \mathbb{B} \GL_d$ representing to the cotangent bundle $T_X^*$. Here, we denote by $\mathbb{B} \GL_d$ the classifying stack of $\GL_d$.

The construction of the Hitchin morphism derives from $G$-invariant functions on $\fC^d_G$. Studying $G$-invariant functions on $\fC^d_G$ amounts to investigate the morphism
\begin{equation} \label{characteristic}
	[\fC^d_G/ G] \to \fC^d_G\sslash G
\end{equation}
between quotients of the commuting scheme $\fC^d_G$ by the diagonal action of $G$
in the sense of algebraic stacks and geometric invariant theory respectively.
By definition, the categorical quotient $\fC^d_G\sslash G$ is the affine scheme whose ring of functions is the $k$-algebra
$$k[\fC^d_G\sslash G]=k[\fC^d_G]^G$$
of $G$-invariant functions on $\fC^d_G$.

The commuting scheme $\fC^d_G$ has been studied intensively, especially in the case $d=2$. It has a non-empty open locus $\fC^{d,\rss}_G$ consisting of commuting linear maps $\theta:V_d\to \fg$ such that the image $\theta(V_d)$ has non-empty intersection with the regular semi-simple locus $\fg^{\rss}$ of $\fg$. This open locus is smooth. In the case $d=2$, Richardson \cite{Richardson} proved that the underlying topological space of $\fC^2_G$ is irreducible, in particular, the locus $\fC^{2,\rss}_G$ is dense in $\fC^2_G$. Results of Iarrobino \cite{L} on punctual Hilbert schemes on $\A^d$, with $d\geq 3$, imply that irreducibility is no longer true for $d\geq 3$.

There is a long-standing conjecture saying that the commuting scheme $\fC^2_G$ is reduced. The generalization of this conjecture to the cases $d\geq 3$ seems to be rather doubtful since 
we have
very little understanding of other components of $\fC^{d}_G$ other than the component containing $\fC^{d,\rss}_G$.
%there may be components of $\fC^d_G$ consisting of only nilpotent matrices which are likely not reduced.   

The categorical quotient $\fC^d_G\sslash G$ behaves better. In \cite{Hunziker}, Hunziker proved a weak version of the Chevalley restriction theorem for the commuting scheme. If $\ft$ is a Cartan subalgebra of $\fg$, the embedding $\ft^d\to \fg^d$ factors through $\fC^d_G$ since $\ft$ is commutative. Since orbits of the diagonal actions of $W$ on $\ft^d$ are contained in orbits of the diagonal action of $G$ of $\fC^d_G$, the restriction of a $G$-invariant function on $\fC^d_G$ to $\ft^d$ is $W$-invariant. In other words, we have a morphism
\begin{equation}\label{Chevalley-Cd}
	\ft^d\sslash W \to \fC^d_G \sslash G.
\end{equation}  
Based on fundamental result of Richardson \cite{Richardson}, Hunziker proved that this morphism is a universal homeomorphism, i.e., it is a finite morphism inducing a bijection on  $k$-points, 
see
\cite[Theorem 6.2, Theorem 6.3]{Hunziker}\footnote{In \cite[Section 6]{Hunziker},
Hunziker works with the reduced quotient $R^{\on{red}}$ of the ring $R$ of functions on $\fC^d_G$. As we are over $k=\bC$, the Reynolds operator implies that there exists an isomorphism between $(R^G)^{\on{red}}$ and $(R^{\on{red}})^G$ for any $k$-algebra of finite type with $G$-action (see, e.g., \cite[page 29]{Mumford}).
Thus Hunziker proves that $\ft^d\sslash W \to (\fC^d_G\sslash G)^{\on{red}}$ is a universal homeomorphism. This is equivalent to saying that $\ft^d\sslash W \to \fC^d_G\sslash G$ is a universal homeomorphism.}. 
In particular, $\ft^d\sslash W$ is the normalization of 
the underlying reduced subscheme $(\fC^d_G \sslash G)^{\on{red}}$. 
Since $\ft^d\sslash W$ is irreducible, the categorical quotient $\fC^d_G \sslash G$ is also irreducible.

\begin{conjecture} \label{reduced-normal-GIT}
 The morphism \eqref{Chevalley-Cd} is an isomorphism.
\end{conjecture}

We note that Conjecture \ref{reduced-normal-GIT} is equivalent to the asserting that the categorical quotient $\fC^d_G\sslash G$ is reduced and normal. Indeed, since $\ft^d\sslash W$ is obviously reduced and normal, if \eqref{Chevalley-Cd} is an isomorphism then $\fC^d_G \sslash G$ also is reduced and normal. Conversely, if $\fC^d_G \sslash G$ is reduced and normal, then the map \eqref{Chevalley-Cd}, known to be a normalization, has to be an isomorphism. Note also that Conjecture \ref{reduced-normal-GIT} together with~\eqref{characteristic}
%that \eqref{Chevalley-Cd} being an isomorphism, 
imply that 
there is a $G$-invariant morphism
\begin{equation}\label{spectral-data-map}
	\on{sd}:\fC^d_G \to \ft^d \sslash W
\end{equation}
to be called the {\em universal spectral data morphism},
making the following diagram commute:
\begin{equation} \label{characteristic-diagram}
\begin{tikzcd}
\ft^d \arrow{r} \arrow{d}
& \fC^d_G \arrow{d} \arrow{ld} \\
\ft^d \sslash W \arrow{r}
& \fC^d_G \sslash G
\end{tikzcd}
\end{equation}
As the existence of this morphism would be important to the study of the Hitchin morphism,  we will state a conjecture, which is a weaker form of Conjecture \ref{reduced-normal-GIT}.

\begin{conjecture}\label{reduced-GIT}
	There exists a $G$-invariant morphism $\on{sd}:\fC^d_G \to \ft^d \sslash W$ making the diagram \eqref{characteristic-diagram} commute.
\end{conjecture}

We note that Conjecture \ref{reduced-GIT} implies that the categorical quotient $\fC^d_G \sslash G$ is reduced. Indeed, the right triangle of \eqref{characteristic-diagram} gives rise to a commutative triangle of rings, which says that the composition of homomorphisms
$$ k[\fC^d_G]^ G \to k[\ft^d]^W \to k[\fC^d_G]$$
is the inclusion map. It follows that the homomorphism 
$k[\fC^d_G]^ G \to k[\ft^d]^W$ is injective. Since $k[\ft^d]^W$ in an integral domain,  $k[\fC^d_G]^ G$ is also an integral domain, and in particular reduced. 

In the next section, Theorem \ref{B'}, we will construct a canonical map $\fC^d_G(k) \to \ft^d \sslash W(k)$ making the diagram \eqref{characteristic-diagram} commute on the level of $k$-points. For the moment, let us construct this map in the case $G=\GL_n$. A $k$-point $\theta\in\fC^d_G(k)$ consists of commuting family of endomorphisms $\theta_1,\ldots,\theta_d$ on the standard $n$-dimensional $k$-vector space $k^n$. It equips with $k^n$ a structure of module over the polynomial algebra $\Sym(V_d)=k[v_1,\ldots,v_d]$, where 
$v_i$ acts by $\theta_i$. Let $F$ denote the corresponding finite $\Sym(V_d)$-module. We have a decomposition $F=\bigoplus_{\alpha\in\A^d} F_\alpha$ where $F_\alpha$ is a $\Sym(V_d)$-module annihilated by some power of the maximal ideal $\fm_\alpha$ corresponding to the point $\alpha\in \A^d(k)$ where $\A^d=\Spec(\Sym(V_d))$. This decomposition gives rise to a $0$-cycle 
$$z(\theta)=\sum_{\alpha\in \A^d(k)} \lg(F_\alpha)\alpha.$$
of length $n$ in $\A^d$. This construction gives rise to a $G(k)$-invariant map $\fC^d_G(k) \to \Chow_n(\A^d)(k)$ where 
$$\Chow_n(\A^d)=\Spec ((\Sym(V_d)^{\otimes n})^{\fS_n}).$$
As $G=\GL_n$, one can identify $\Chow_n(\A^d)$ with $\ft^d \sslash W$ and we thus obtain 
the desired map from $\fC^d_G(k)$ to $\ft^d\sslash W(k)$.
We shall show that the construction above works in families.

\begin{theorem}\label{construction of sd}
	Conjecture \ref{reduced-GIT} holds in the case of $\GL_n$. In particular, for $G=\GL_n$, the categorical quotient $\fC^d_G\sslash G$ is reduced. 
\end{theorem}

\begin{proof}
The construction of the canonical map 
$\sd:\fC_G^d\to \ft^d \sslash W=\Chow_n(\A^d)$ in the case $G=\GL_n$ is due to Deligne \cite[Section 6.3.1]{Deligne}. For reader's convenience, we will recall his construction. For any $k$-algebra $R$,
we will construct a functorial map $\fC^d_G(R)\to \Chow_n(\A^d)(R)$ following Deligne. A collection of $d$ matrices $\alpha_1,\ldots,\alpha_d \in\fg(R)=\gl_n(R)$ gives rise to a $k$-linear map $\alpha:V_d \to \fg(R)$. If $\alpha_1,\ldots,\alpha_d$ commute with each other, $\alpha$ gives rise to a homomorphism of $k$-algebras \[\Sym(\alpha):\Sym(V_d)\to \fg(R).\] 
By composing with the determinant, we get a map $\det\circ\Sym(\alpha):\Sym(V_d)\to R$ which is a homogenous algebraic map of degree $n$ on the infinite dimensional vectors space $\Sym(V_d)$. It must derive from a polynomial linear map 
\beq\label{z(alpha)}
z(\alpha):(\Sym(V_d)^{\otimes n})^{\fS_n} \to R
\eeq
characterized by the property that 
\[z(\alpha)(f^{\otimes n})=\det\circ\Sym(\alpha)(f)\]
for $f\in\Sym(V_d)$.
Since $\det\circ\Sym(\alpha)$ is multiplicative, $z(\alpha)$ is a homomorphism of $k$-algebras. In other words, $z(\alpha)$ defines a $R$-point of $\Chow_n(\A^d)$.	
This finishes the construction of the map $\sd:\fC_G^d\to \ft^d \sslash W$.

We shall prove that the compostion 
$\frak C^d_G\stackrel{\sd}\to\ft^d\sslash\rW\stackrel{\eqref{Chevalley-Cd}}\to\frak C^d_G\sslash G$ is the quotient map.
Equivalently, the induced map 
\[
k[\frak C^d_G]^G\rightarrow k[\frak t^d]^\rW\stackrel{}\rightarrow k[\frak C^d_G]
\]
on rings of functions 
is the natural inclusion map.
%%%%%%%
\quash{
Consider the following diagram
\beq\label{d:first reduction}
\xymatrix{&\frak C^d_G\ar[r]\ar[d]\ar[ld]_\chi&\frak g^d\ar[d]\\
\ft^d/\rW\ar[r]^q&\frak C^d_G/G\ar[r]^s&\frak g^d/G}
\eeq
As the right square diagram in \eqref{d:first reduction} is commutative and $s$ is a closed embedding, it suffices to show that the outer diagram 
in \eqref{d:first reduction} is commutative.
We will prove the desired commutativity by showing that the following diagram 
\beq\label{d:second reduction}
\xymatrix{k[\frak C^d_G]&k[\frak g^d]\ar[l]\\
k[\ft^d]^\rW\ar[u]^h&k[\frak g^d]^G\ar[l]\ar[u]}
\eeq
is commutative. Here the horizontal arrows are the restriction maps 
induced by the embeddings $\ft^d\to\fg^d$, $\fC^d_G\to\fg^d$
and the right vertical arrow is the natural embedding.
}
%%%%%%%%%%%%%%%
%Consider the natural isomorphisms 
%\[k[\ft^d]\is k[T(i)_{j}]_{1\leq i\leq d,1\leq j\leq n},\ \ k[\fg^d]\is k[X(i)_{a,b}]_{1\leq i\leq d, 1\leq a,b\leq n},\] where 
%$T(i)_{j}$ and $X(i)_{a,b}$ Let $X(i)\in\fg(k[\fg^d])$ be the $n\times n$

Let $X(i)\in\fg(k[\fg^d])$ be the $n\times n$ 
matrix whose $(a,b)$-entry is given 
by the coordinate function for the $(a,b)$-entry of the 
$i$-th copy of $\fg$ in $\fg^d$.
The embedding $\ft^d\to\fg^d$ gives rise to a map
$\fg(k[\fg^d])\to\fg(k[\ft^d])$ and we define 
$T(i)\in \fg(k[\ft^d])$ to be the image of $X(i)$ under this map.
We use the same notation  
$X(i)\in \fg(k[\fC_G^d])$ for the image of $X(i)$ under the natural
map $\fg(k[\fg^d])\to\fg(k[\fC_G^d])$.
It is known that (see, e.g., \cite{P}) the ring of $G$-invariant functions 
$k[\fg^d]^G$ is 
generated by 
\[\on{Tr}(X(i_1)\cdot\cdot\cdot X(i_k))\]
 where $k\in\bZ_{\geq 0}$ and $1\leq i_1,...,i_k\leq d$.
As the restriction map $k[\fg^d]^G\to k[\fC_G^d]^G$ is surjective and 
$[X(i),X(j)]=0\in\fg(k[\fC_G^d])$, it follows that 
$k[\fC^d_G]^G$ is generated by the 
$G$-invariant functions \[\on{Tr}(X(1)^{a_1}\cdot\cdot\cdot\ X(d)^{a_d})\]
where $a_j\in\bZ_{\geq 0}$.
It is easy to see that the image of 
$\on{Tr}(X(1)^{a_1}\cdot\cdot\cdot X(d)^{a_d})$
under the map $k[\fC^d_G]^G\to k[\ft^d]^\rW$ is equal to
\[\on{Tr}(T(1)^{a_1}\cdot\cdot\cdot T(d)^{a_d}).\]
Thus to prove the desired claim, it suffices to show that 
\beq\label{e:key}
z(\alpha)(\on{Tr}(T(1)^{a_1}\cdot\cdot\cdot T(d)^{a_d})=
\on{Tr}(X(1)^{a_1}\cdot\cdot\cdot X(d)^{a_d})
\eeq
where $z(\alpha)=\sd^*:k[\ft^d]^\rW=(S(V_d)^{\otimes n})^{\fS_n}\to k[\fC^d_G]$
is the map in~\eqref{z(alpha)} in the universal case: $R=k[\fC^d_G]$
and $\alpha:V_d\to\fg(R)$
corresponds to the identity map $\id\in\fC^d_G(R)$.

Let $v_1,...,v_d$ be the coordinate vectors of $V_d$.
We have 
\beq\label{iden}
S(\alpha)(v_i)=X(i)\in \fg(k[\fC^d_G])
\eeq
For any $x\in k$ consider the element 
$f=x-v_1^{a_1}\cdot\cdot\cdot v_d^{a_d}\in S(V_d)=k[v_1,...,v_d]$.
It follows from the definition of $z(\alpha)$ that 
\[
z(\alpha)(f^{\otimes n})=
\on{det}\circ S(\alpha)(f)=\on{det}(x\id-S(\alpha)(v_1^{a_1})\cdot\cdot\cdot S(\alpha)(v_d^{a_d}))=
\]
\beq\label{e:tilde h}=
x^n-\on{Tr}(S(\alpha)(v_1^{a_1})\cdot\cdot\cdot S(\alpha)(v_d^{a_d}))x^{n-1}+\cdot\cdot\cdot
.
\eeq
On the other hand, under the canonical identification
$(S(V_d)^{\otimes n})^{\fS_n}=k[\ft^d]^\rW$,
the element $f^{\otimes n}$ becomes 
\[\on{det}(x\id-T(1)^{a_1}\cdot\cdot\cdot T(d)^{a_d})\]
and it follows that 
\[
z(\alpha)(f^{\otimes n})=z(\alpha)(\on{det}(x\id-T(1)^{a_1}\cdot\cdot\cdot T(d)^{a_d}))=\]
\beq\label{e:h}
=x^n-z(\alpha)(\on{Tr}((T(1)^{a_1}\cdot\cdot\cdot T(d)^{a_d})))x^{n-1}+\cdot\cdot\cdot.
\eeq
Comparing the coefficients of $x^{n-1}$ in \eqref{e:tilde h} and \eqref{e:h}, we obtain
\beq\label{e:T=theta}
z(\alpha)(\on{Tr}(T(1))^{a_1}\cdot\cdot\cdot(T(d))^{a_d})=
\on{Tr}(S(\alpha)(v_1)^{a_1}\cdot\cdot\cdot S(\alpha)(v_d)^{a_d}),
\eeq
and it implies
\[
z(\alpha)(\on{Tr}(T(1))^{a_1}\cdot\cdot\cdot(T(d))^{a_d})\stackrel{\eqref{e:T=theta}}=
\on{Tr}(S(\alpha)(v_1)^{a_1}\cdot\cdot\cdot S(\alpha)(v_d)^{a_d})\stackrel{\eqref{iden}}=
\on{Tr}( X(1)^{a_1}\cdot\cdot\cdot X(d)^{a_d}).
\]
Equation \eqref{e:key} follows.
This completes the proof of the proposition.
\end{proof}

Although we don't know the validity of Conjectures \ref{reduced-normal-GIT} and \ref{reduced-GIT} in general, we know they are true on the level of topological spaces. This will allow us to work around and predict the image of the Hitchin map. 

\begin{remark}
In \cite{GG}, Gan and Ginzburg proved the reducedness of 
$\fC_G^d\sslash G$ in the case $G=\GL_n$, $d=2$, by a different method.

\end{remark}

\section{Weyl's polarization and the Hitchin morphism}
Weyl's polarization is a method to construct $G$-invariant functions on the space $\fg^d$ of $d$ arbitrary elements $\theta_1,\ldots,\theta_d \in \fg$. 
The idea is as follows.
Given a $G$-invariant function $c$ on $\fg$ and $x_1,...,x_d\in k$, the map
$$(\theta_1,\ldots,\theta_d)\mapsto c(x_1\theta_1+\cdots+x_d\theta_d)$$
defines a $G$-invariant function on $\fg^d$.
%If we treat $x_1,...,x_d$ as variables, then  $c(x_1\theta_1+\cdots+x_d\theta_d)$
%defines an element in the polynomial ring 
%$k[\fg^d]^G[x_1,...,x_d]$ over the ring $k[\fg^d]^G$ of $G$-invariant functions on 
%$\fg^d$ and
 %its coefficients are called the polarizations of 
%$c$. 
%If $c$ is homogenous of degree $e$, then
%$c(x_1\theta_1+\cdots+x_d\theta_d)$ 
%is a symmetric form of degree $e$ in the variables $x_1,\ldots,x_d$.
Although those $G$-invariant functions on $\fg^d$  in general may not generate $k[\fg^d]^{G}$ (see, e.g., \cite{LMP}), as we shall see, they are close to forming a set of generators of the ring $k[\fC_{G}^d]^{G}$ of $G$-invariant functions on the commuting scheme, and they do in the case $G=\GL_n$.

We will formalize the construction above as follows. %Let $V_d$ denote the typical $d$-dimensional $k$-vector space. 
For every affine variety $Y$ equipped with an action of $\Gm$, 
the functor on the category of $k$-algebras which associates with each $k$-algebra $R$ the set of $\Gm$-equivariant maps $V_d\otimes_k R \to Y$
is representable by an affine scheme, denoted by 
$Y_{\bG_m}^{V_d}$.
For instance, if $Y$ is the affine line $\bA^1=\Spec (k[x])$ equipped with an action of $\Gm$ given by $t.x=t^e x$, then $Y^{V_d}_{\Gm}$ is the $e$-th symmetric tensor of $\A^d=\Spec(\Sym(V_d))$. For $Y=\fg$, the space $\fg^{V_d}_{\Gm}$ can be identified with $\fg^d$. Let us also consider the case $Y=\fc$ where $\fc=\fg\sslash G$. Since $\fc$ is isomorphic to an $n$-dimensional affine space with homogenous coordinates $c_1,\ldots,c_n$ of degree $e_1,\ldots,e_n$, the space $A=\fc^{V_d}_{\Gm}$ 
%of $\Gm$-equivariant maps $V_d\to \fc$
is isomorphic to:
\begin{equation}\label{A}
	A \simeq \prod_{i=1}^n\Sym^{e_i}\A^d.
\end{equation}
The isomorphism depends on the choice of homogenous coordinates $c_1,\ldots, c_n$. 

Since the morphism $\fg\to \fc$ is $G$-invariant and $\Gm$-equivariant, it induces a $G$-invariant morphism 
\begin{equation}\label{polarization}
	\pol: \fg^d \to A\simeq \prod_{i=1}^n\Sym^{e_i}\A^d
\end{equation}
which embodies Weyl's polarization method for the diagonal action of $G$ on $\fg^d$. 
For example, in case $G=\GL_n$, given $d$ arbitrary matrices $\theta=(\theta_1,...,\theta_d)\in(\gl_n)^d$,
the trace of the $i$-th power of 
$x_1\theta_1+\cdots+x_d \theta_d$ is an $i$-th symmetric form in the variables $x_1,...,x_d$ and thus 
defines a point $\pol_i(\theta)$ in $\Sym^i\A^d$ and we have 
$\pol(\theta)=(\pol_1(\theta),...,\pol_n(\theta))$.
 Instead of using trace of powers of an endomorphism, we may also use 
the homogenous coordinates of $\fc$ given by the $i$-th coefficient of the characteristic polynomial of an endomorphism for $1\leq i\leq n$.
The latter invariant function is used by Simpson to define the Hitchin morphism for $\GL_n$ for higher dimensional varieties \cite{Simpson 2}. We have seen that the choice of coordinates of $\fc$ is unimportant as it just gives rise to different isomorphisms \eqref{A}.
 
By restricting \eqref{polarization} to the commuting scheme $\fC^d_G$, we obtain a $G$-invariant morphism 
\begin{equation} \label{h}
	h:\fC^d_G \to A.
\end{equation}
To study the structure of the Hitchin morphism, and in particular the image thereof, we need to understand the image of the map~\eqref{h} and its relation to the Chevalley restriction morphism \eqref{Chevalley-Cd}. 
For that purpose, we will also need to use Weyl's polarization construction for the diagonal action of $W$ on $\ft^d$. The morphism $\ft\to \fc=\ft\sslash W$ is $W$-invariant and $\Gm$-equivariant. As a result, we have a $W$-invariant morphism
\begin{equation}\label{polW}
	\pol_W:\ft^d \sslash W \to A.
\end{equation}
We recall the following \cite[Theorem 2.15]{LMP}:
\begin{theorem}\label{B}
	The morphism $\pol_W$ of \eqref{polW} is finite and induces an injective map on $k$-points. In other words, there exists a unique reduced closed subscheme $B$ of $A$ such that $\pol_W$ factors through a morphism 
	\begin{equation} \label{b}
		b: \ft^d \sslash W\to B,
	\end{equation}
	which is a universal homeomorphism and normalization. For $G=\GL_n$, $\pol_W$ is a closed embedding and $b$ is an isomorphism. 
\end{theorem}

\begin{remark}
In the case $G=\GL_n$,
%$\pol_W$ induces an isomorphism of $\ft^d \sslash W$ on a closed subscheme of $A$, which is necessarily reduced since $\ft^d$ is reduced. 
the theorem above 
is  the first fundamental theorem for symmetric groups, which is a classical theorem of Weyl \cite[II.A.3]{Weyl}. 
According to Hunziker \cite{Hunziker}, $\pol_W$ is a closed embedding for type B,C.	
According to Wallach \cite{Wallach}, $\pol_W$ fails to be a closed embedding for type D.  
\end{remark}	

\begin{example}
	Let us describe the closed subscheme $B$ of $A$ in in the case $G=\SL_2$ and $d=2$. 
In this case the Cartan algebra can be identified with $\ft\simeq \Spec (k[t])$. The Weyl group $W=\fS_2$ on $\ft$ by $w(t)=-t$ where $w$ is the non-trivial element of $W$. The categorical quotient $\fc=\Spec (k[u])$ with $u=t^2$ and the morphism $\fg\to \fc$ is given $u=\det(g)$. Since the exponent $e=2$, we have $A=\Sym^2\A^2$ which is a 3-dimensional vector space. The map $\ft^2=\A^2 \to A=\Sym^2(\A^2)$ is given by $v\mapsto v^2$. In coordinates, this is the map $\A^2\to \A^3$ given by $(x,y)\mapsto (x^2,2xy,y^2)$. Thus $B$ is the closed subscheme of $A=\A^3$ defined by the equation $b^2-4ac=0$ which can be identified with the categorical quotient of $\A^2$ by the action of $\fS_2$ given by $(x,y)\mapsto (-x,-y)$.
\end{example}
	
We have the following factorization of $h:\fC^d_G \to A$:
\begin{theorem} \label{B'}
There exists a closed subscheme $B'$ of $A$, which is a thickening of the closed subscheme $B$ of $A$, as in Theorem \ref{B}, such that the morphism $h:\fC^d_G \to A$ in \eqref{h} factors through a morphism
\begin{equation} \label{sd'}
\on{sd}':\fC^d_G \to B'.
\end{equation}
In particular, there is a canonical 
$G(k)$-equivariant morphism
$\fC^d_G(k) \to \ft^d\sslash\rW(k)$.
For $G=\GL_n$, we have $B'=B$ and~\eqref{sd'}
is equal to the universal spectral data morphism
$\sd:\fC^d_G\mapsto \ft^d\sslash\rW\is B$
constructed in Theorem \ref{construction of sd}.
\end{theorem}

\begin{proof}
By \cite[Theorem 6.3]{Hunziker}, the Chevalley restriction map
$\ft^d\sslash\rW\to\fC^d_G \sslash G$ is a homeomorphism. 
Therefore, the diagram~\eqref{characteristic-diagram} implies that 
the $G$-invariant morphism $h:\fC^d_G\to A$ factors through a 
thickening $B'$ of the closed subscheme B of A. The first claim follows.
The second claim follows from Theorem \ref{construction of sd}.
%%%%%%%%%%
\quash{
We first prove this assertion on the level of $k$-points. This amounts to prove that a $k$-point of $A$ lies in the image of the morphism $h$ of \eqref{h} if and only if it lies in the image of the morphism $\pol_W$ of \eqref{polW}. Since the Cartan algebra $\ft$ is commutative, $\ft^d$ is contained in $\fC^d_G$ and therefore its image in $A$ is contained in the image of $\fC^d_G(k)$. In the opposite direction, it follows from Richardson's theorem \cite[3.6]{Richardson} that the $G$-orbit passing by a point $(x_1,\ldots,x_d)\in \fC^d_G(k)$ is closed if and only if it has non-empty intersection with $\ft^d$. Indeed, Richardson proved that the $G$-orbit passing by $(x_1,\ldots,x_d)\in \fg^d$ is closed if and only if elements $x_1,\ldots,x_d$ generate a reductive Lie algebra. In the case where $x_1,\ldots,x_d$ commute with each other, this means $x_1,\ldots,x_d$ lie in a same Cartan subalgebra (or simultaneously diagonalizable in $\gl_n$ case). Since every orbit has a closed orbit contained in its closure, for every $x\in \fC^d_G(k)$, there exists $y\in \ft^d(k)$ lying in the closure of the $G$-orbit of $x$. Since $x$ is a $G$-invariant map, $x$ and $y$ have the same image in $A$.} 
%%%%%%%%
\end{proof}

One may ask whether Theorem \ref{B'} holds for $B'=B$ for general $G$. This would follow from Conjecture \ref{reduced-GIT}.

\section{Postulated image of the Hitchin morphism and cameral covers}\label{postulated image}

Let $X$ be a proper smooth algebraic variety over $k$ of dimension $d$. A Higgs bundle over $X$ is represented by a map $\theta:X\to [\fC^d_G/(G\times \GL_d)]$ lying over the map $\tau_X^*:X\to {\mathbb B} \GL_d$ given by its cotangent bundle $T_X^*$. By composing it with the map $[h]:[\fC^d_G/(G\times \GL_d)] \to [A/\GL_d]$ derived from~\eqref{h}, we obtain the Hitchin morphism \[h_X:\cM_X\to \cA_X\] where $\cA_X$ is the space of maps $X\to [A/\GL_d]$ lying over $\tau_X^*$. 
By choosing a system of homogenous coordinates $c_1,\ldots,c_n$ of $\fc$ of degrees $e_1,\ldots,e_n$, we can identify $\cA_X$ with the vector space
$\bigoplus_{i=1}^n\rH^0(X,\Sym^{e_i}\Omega_X^1)$. 
%In \cite{Simpson 1}, Simpson constructed the Hitchin for $G=\GL_n$ by using the system of homogenous coordinates $c_1,\ldots,c_n$ of $\fc$ with $c_i$ being the invariant polynomial on $\gl_n$ given by $c_i(x)=\tr(x^i)$. 

Let $\cB_X$ denote the space of maps $X\to [B/\GL_d]$, where $B$ is the closed subscheme of $A$ defined in Theorem \ref{B}, lying over $\tau_X^*$. It is clear that $\cB_X$ is a closed subscheme of $\cA_X$. We call it the {\em postulated image} of the Hitchin map $h_X$. By replacing $B$ by its thickening $B'$ as in Theorem \ref{B'}, we have a thickening $\cB'_X$ of $\cB_X$. 
The schemes $\cB_X$ and $\cB'_X$ have the same underlying topological space.

\begin{proposition}
Let $X$ be a proper smooth algebraic variety of dimension $d$ over an algebraically closed field $k$ of characteristic zero, and let $\cM_X$ be the moduli stack of Higgs bundles over $X$. Then the Hitchin morphism $h_X:\cM_X\to \cA_X$ factors through a
map
$$\sd'_X:\cM_X\to \cB'_X$$
to be called the spectral data morphism.
In particular, the image of every geometric point $\theta\in \cM_X(k)$ under the 
Hitchin morphism 
belongs to $\cB_X(k)$. 
\end{proposition}

\begin{proof}
By Theorem \ref{B'}, for any $S$-point $\theta:S\times X\to [\fC^d_G/(G\times \GL_d)]$ in $\cM_X(S)$ where 
$S$ is a $k$-scheme, its image $h_X(\theta):S\times X \to [A/\GL_d]$ factors through $b': S\times X\to [B'/\GL_d]$. This gives the desired factorization $\sd'_X:\cM_X\to \cB'_X$ of the Hitchin morphism.
	Assume $\theta\in\cM(k)$.  
	%by $h_X$ is a map $a:X\to [A/\GL_d]$ lying over $\tau_X^*:X\to  \mathbb{B}\GL_d$ representing the cotangent bundle of $X$. By Theorem \ref{B'}, the map  $a:X\to [A/\GL_d]$ factors through $[B'/\GL_d]$. 
	Since $X$ is reduced, its image $b': X\to [B'/\GL_d]$
	factors through a morphism $b:X\to [B/\GL_d]$, i.e., we have $h_X(\theta)\in\cB_X(x)$.
	The proposition follows
	\end{proof}
%%%%%%%%%%

\begin{conjecture}\label{Conj surjectivity}
For every $b\in \cB_X(k)$, the fiber $h_X^{-1}(b)$ is non-empty.	
\end{conjecture} 

\begin{example}
Consider the case when $X$ is a $d$-dimensional abelian variety. By choosing an isomorphism between the Lie algebra of $X$ and the typical $d$-dimensional vector space $V_d$, we will have an isomorphism $\cA_X=A$ and $\cB_X=B$ which is a strict subset of $A$ for $d\geq 2$. We can also prove that the spectral data map $\cM_X(k)\to \cB_X(k)$ is surjective by restricting ourselves to the subset of $\cM_X(k)$ consisting of Higgs bundles $(E,\theta)$ where $E$ is the trivial $G$-bundle.
\end{example}

One can think of $\cB_X(k)$ as the subset of $\cA_X(k)$ consisting of points $b\in\cA_X(k)$ for which one can construct a cameral covering. 
For any scheme $Y$ with an action of $\GL_d$, we can form the twist $Y_{T^*_X}$ of $Y$ by the $\GL_d$-torsor given by $T_X^*$. Then
a point $b\in\cB_X(k)$ gives rise to a map
$b:X\to B_{T^*_X}$ and, since the map 
$(\ft^d\sslash W)_{T^*_X}\to B_{T^*_X}$ induced from 
$\ft^d\sslash W\to B$ is the normalization 
and $X$ is normal, the map $b$ lifts to a map $X\to (\ft^d\sslash W)_{T^*_X}$. We define 
$\tilde X_b$ to be the fiber product
 
\[\xymatrix{\tilde X_b\ar[r]\ar[d]&(\ft^d)_{T^*_X}\ar[d]\\
	X\ar[r]&(\ft^d\sslash W)_{T^*_X}.}\]
The projection $\tilde X_b\to X$, which is a finite surjective morphism, is called the cameral covering associated with $b$. 
%gives rise to a morphism $X_b\to X$ as the base change of $(\ft^d\sslash W)_{T^*_X}\to B_{T^*_X}$. By construction, $X_b\to X$ is a universal homeomorphism. By considering further base change of $\ft^d\to \ft^d\sslash W$, we obtain a finite morphism 
%$\tilde X_b \to X_b$ with

Let $B^\circ$ denote the open dense locus of $B$ where the morphism $\ft^d \to B$ is a finite étale Galois morphism with Galois group $W$. This is a $\GL_d$-equivariant open subset of $B$. 
\begin{definition}
We define $\cB_X^\heartsuit(k)$ to be the open locus of $\cB_X(k)$ consisting of maps $b:X\to [B/\GL_d]$ whose image has non-empty intersection with $[B^\circ/\GL_d]$
\end{definition}
For every $b\in \cB_X^\heartsuit(k)$, the cameral covering $\tilde X_b\to X$ is generically a finite étale Galois morphism with Galois group $W$. 
We will prove Conjecture \ref{Conj surjectivity} in the case $G=\GL_n$ and $d=2$ for all $b\in \cB^\heartsuit_X(k)$. In the one-dimensional case, and for $G=\GL_n$, it is well-known that spectral curves are more convenient than cameral curves for the purpose of constructing Higgs bundles. Cameral and spectral covers are generally not flat in higher dimension, but in the case of dimension two, there is a canonical way to make them flat. 

From now on, we will assume $G=\GL_n$.
\section{Spectral covers} 

Let us first review the construction of the universal spectral cover for $d=1$. For the group $\GL_n$, $\ft=\bA^n=\Spec(k[x_1,\ldots,x_n])$ is the space of diagonal matrices with entries $x_1,\ldots,x_n$. 
The Weyl group $W$ is the symmetric group $\fS_n$ acting on $\bA^n$ by permutation of coordinates $x_1,\ldots,x_n$. 
By the fundamental theorem of symmetric polynomials, the categorical quotient $\fc=\bA^n\sslash \fS_n$ is the affine space of coordinates
\begin{eqnarray*}\label{symmetric-poly}
	c_1 & = & x_1+\cdots+x_n,\\
	c_2 & = & x_1x_2+x_1x_3+\cdots+x_{n-1}x_n, \\
	&\cdots& \\
	c_n & = & x_1\ldots x_n.
\end{eqnarray*}
The universal spectral cover is a finite flat covering $\fc^\bullet\to\fc$ of degree $n$. 
To construct it we consider the action of the subgroup $\fS_{n-1}$ of $\fS_n$ on $\bA^n$ permuting the coordinates $(x_1,\ldots,x_{n-1})$ and leaving $x_n$ fixed. The categorical quotient $\fc^\bullet=\bA^n\sslash\fS_{n-1}$ is the affine space of coordinates $(c'_1,\ldots,c'_{n-1},x_n)$ with 
$$c'_1=x_1+\cdots+x_{n-1},\ldots,c'_{n-1}= x_1\ldots x_{n-1}.$$ 
The induced morphism $p:\fc^\bullet\to \fc$ is a finite flat morphism of degree $n$. One can represent the finite morphism $\fc^\bullet\to \fc$ in terms of equations by considering the morphism $\iota:\fc^\bullet\to \fc\times \A^1$ given $(c'_1,\ldots,c'_{n-1},x_n) \mapsto (c_1,\ldots,c_n,t)$ with  
\begin{equation} \label{bullet-equation}
	t=x_n, c_1=c'_1+x_n, c_2=c'_2+c'_1 x_n,\ldots, c_n=c'_{n-1}x_n.
\end{equation}
This is a closed embedding that identifies $\fc^\bullet$ with the closed subscheme of $\fc\times\A^1$ defined by the equation $t^n-c_1 t^{n-1}+\cdots + (-1)^n c_n=0$.

We will now generalize this construction to the case $d\geq 2$. For $G=\GL_n$, we have $\ft^d=(\A^d)^n$. The categorical quotient $\ft^d \sslash W$ can be identified with the Chow scheme $\Chow_n(\A^d)=(\A^d)^n\sslash \fS_n$ classifying zero-dimensional cycles of length $n$ of $\A^d$.
We will represent a point of $\Chow_n(\A^d)$ as an unordered collection of $n$ points of $\A^d$
\begin{equation}
	[x_1,\ldots,x_n]\in \Chow_n(\A^d).
\end{equation} 
By 
Theorem \ref{B}, the morphism
\begin{equation}\label{WeylFFT}
	\on{pol}_\rW:\Chow_n(\A^d) \to \prod_{i=1}^n \Sym^i\A^d,\ \ 
	[x_1,...,x_n]\to (c_1,...,c_n)
\end{equation}
where $c_i\in\Sym^i\A^d$ is the $i$-th elementary symmetric polynomial of variables $x_1,\ldots,x_n\in\A^d$,
is a closed embedding. 
We will construct the universal spectral covering of $\Chow_n(\A^d)$ as follows. 
Consider the morphism
\begin{equation}
	\chi_{\A^d}:\Chow_n(\A^d) \times \A^d \to\Sym^n\A^d
\end{equation}
given by 
\begin{equation}
\chi_{\A^d}([x_1,\ldots,x_n],x)= (x-x_1)\ldots(x-x_n)=x^n-c_1 x^{n-1}+\cdots+(-1)^n c_n
\end{equation}
We define the closed subscheme $\on{Cayley}_n(\A^d)$ to be 
\begin{equation}
	\on{Cayley}_n(\A^d)=\chi_{\A^d}^{-1}(\{0\})
\end{equation}
the fiber over $0\in\Sym^n\A^d$. 

\begin{proposition} \label{Cayley}
\begin{enumerate}
\item The projection $p:\on{Cayley}_n(\A^d)\to\Chow_n(\A^d)$ is a finite morphism which is étale over the open subset $\Chow^\circ_n(\A^d)$ of $\Chow_n(X)$ consisting of multiplicity free $0$-cycles.  
\item  For every point $a=[x_1^{n_1},\ldots,x_m^{n_m}]\in \Chow_n(\A^d)$ where $x_1,\ldots,x_m$ are distinct points of $\A^d$, and $n_1,\ldots,n_m$ are positive integers such that $n_1+\cdots+n_m=n$, the fiber of $p:\on{Cayley}_n(\A^d)\to\Chow_n(\A^d)$ over $a$ is the finite subscheme of $\A^d$
	\begin{equation}
		\on{Cayley}_n(a)=\bigsqcup_{i=1}^m \Spec(\cO_{\A^d,x_i}/\fm_{x_i}^{n_i}),
	\end{equation}
where $\cO_{\A^d,x_i}$ is the local ring of $\A^d$ at $x_i$, and $\fm_{x_i}$ its maximal ideal. In particular, as soon as $d\geq 2$ and $n\geq 2$, then the cover $\on{Cayley}_n(\A^d)\to \Chow_n(\A^d)$ is not flat.

\item Let $F$ be a finite $\cO_{\A^d}$-module of length $n$ and let $a\in\Chow_n(\A^d)$ be its spectral datum. Then $F$ is supported by the finite subscheme $\on{Cayley}_n(a)$ of $\A^d$ (This is a generalization of the Cayley-Hamilton theorem).
\end{enumerate}	
\end{proposition}

\begin{proof}
We will first describe a set of the generators of the ideal defining the closed subscheme $\on{Cayley}_n(\A^d)$
of $\Chow_n(\A^d)\times\A^d$. Let $V_d$ be the space of linear forms on $\A^d$.
Every $v:\A^d\to\A^1$ in $V_d$ induces a map on Chow varieties $[v]:\Chow_n(\A^d)\to \Chow_n(\bA^1)$ mapping $a=[x_1,\ldots,x_n]\in\Chow_n(\A^d)$ to 
$$v(a)=[v(x_1),\ldots,v(x_n)]\in \Chow_n(\bA^1).$$ 
As the diagram 
\begin{equation}
	\begin{tikzcd}
\Chow_n(\A^d) \times \A^d \arrow{r}{\chi_{\A^d}} \arrow{d}[swap]{[v]\times v}
& \Sym^n\A^d \arrow{d}{\Sym^n(v)} \\
\Chow_n(\bA^1) \times \bA^1 \arrow{r}[swap]{\chi_{\bA^1}}
& \Sym^n\bA^1=\bA^1
\end{tikzcd}
\end{equation}
is commutative, the function $f_v=\chi_{\bA^1}\circ ([v]\times v):\Chow_n(\A^d) \times \A^d \to \bA^1$ vanishes on $\on{Cayley}_n(\A^d)$. Explicitly for every $a=[x_1,\ldots,x_n]\in \Chow_n(\A^d)$, we have
\begin{equation} \label{fta}
	f_v(a,x)=(v(x)-v(x_1)) \ldots (v(x)-v(x_n)).
\end{equation}
Moreover, for $\Sym^n(v)$ generates the ideal defining $0$ in $\Sym^n\A^d$ as $v$ varies in $V_d$, the functions $f_v$ generate the ideal defining $\on{Cayley}_n(\A^d)$ inside $\Chow_n(\A^d)\times \A^d$. This provides a convenient set of generators of this ideal albeit infinite and even innumerable as $k$ may be. 

\begin{enumerate}
\item Let $v_1,\ldots,v_d$ be the standard basis of $V_d$ 
whose symmetric algebra $\Sym(V_d)$ is the ring of functions of $\A^d$.
The functions $f_{v_1},\ldots,f_{v_d}$ cut out a closed  subscheme $Z$ of $\Chow_n(\A^d)\times \A^d$ which is finite flat of degree $n^d$ over $\Chow_n(\A^d)$. Since $\on{Cayley}_n(\A^d)$ is a closed subscheme of $Z$, it is also finite over $\Chow_n(\A^d)$. This proves the first assertion of the proposition.

\item We will prove that for $a=[x_1^{n_1},\ldots,x_m^{n_m}]\in \Chow_n(\A^d)$ where $x_1,\ldots,x_m$ are distinct points of $\A^d$, and $n_1,\ldots,n_m$ are positive integers such that $n_1+\cdots+n_m=n$, $\on{Cayley}_n(a)$ is the closed subscheme of $\A^d$ defined by the ideal $\fm_{x_1}^{n_1} \ldots \fm_{x_m}^{n_m}$ of $\Sym(V_d)$ where $\fm_{x_i}$ is the maximal ideal corresponding to the point $x_i\in \A^d$. 

Let us denote $I_a$ the ideal of $\Sym(V_d)$ defining the finite subscheme $\on{Cayley}_n(a)$ in $\A^d$. We first prove that $I=I_{x_1}\ldots I_{x_n}$ where $A/I_{x_i}$ is supported by some finite thickening of the point $x_i$. For this we only need to prove that for every $x\notin\{x_1,\ldots,x_m\}$, there exists a function $f\in I_a$ such that $f\notin\fm_x$. We recall that the ideal $I_a$ is generated by the functions $f_v(a):\A^d\to\bA^1$ as $v$ varies in $V_d$. Choose a linear form $v\in V_d$ a linear form on $\A^d$ such that $v(x)\neq v(x_i)$ for all $i\in\{1,\ldots,m\}$, then we have $f_v(a)(x)\neq 0$ by \eqref{fta}. 

As $x_1,\ldots,x_m$ play equivalent roles, we can focus our attention on $x_1$. It  only remains to prove that the images of the functions $f_v(a)$ in the localization $\Sym(V_d)_{x_1}$ of $\Sym(V_d)$ at $x_1$, as $v$ varies in $V_d$, generate the ideal $\fm_{x_1}^{n_1}$. From \eqref{fta}, we already know that $f_v(a)\in \fm_{x_1}^{n_1}$ for every $v\in V_d$. By the Nakayama lemma, we only need to prove that the images of $f_v(a)$ in $\fm_{x_1}^{n_1}/\fm_{x_1}^{n_1+1}$ generate this vector space as $v$ varies in $V_d$. We observe that for $v\in V_d$ such that $v(x_1)\neq v(x_i)$ for $i\in \{2,\ldots,m\}$, the factors $v(v)-v(x_2),\ldots v(v)-v(v_m)$ are all invertible at $x_1$, it is enough to prove that for $v\in V_d$ satisfying the open condition $v(x_1)\neq v(x_i)$ for $i\in \{2,\ldots,m\}$, the functions $(v(v)-v(x_1))^{n_1}$ generate $\fm_{x_1}^{n_1}/\fm_{x_1}^{n_1+1}$. Here we use again the fact the image of the $n$-th power map $\fm_x/\fm_x^2 \to \fm_x^n/\fm_x^{n+1}$ span $\fm_x^n/\fm_x^{n+1}$ and this conclusion doesn't change even after we remove from $\fm_x/\fm_x^2$ a closed subset of smaller dimension.

\item By the Chinese remainder theorem we are easily reduced to prove that if $F$ is a finite $\Sym(V_d)$-module of length $n$, supported by a finite thickening of $x\in \A^d$ then $F$ is annihilated by $\fm_x^n$. Since $F$ is  supported by a finite thickening of $x\in\A^d$ it has a structure of $\Sym(V_d)_x$-module where $\Sym(V_d)_x$ is the localization of $\Sym(V_d)$ at $x$. We consider the decreasing filtration $F \supset \fm_x F \supset \fm_x^2 F \supset \cdots$.  By the Nakayama lemma, we know that for $m\in\N$, $\fm_x^m E/\fm_x^{m+1}E=0$ implies $\fm_x^m F=0$. It follows that as long as $\fm_x^m F\neq 0$, we have $\dim_k(\fm_x^i F/\fm_x^{i+1}F)\geq 1$ for all $i\in \{0,\ldots,m\}$ and it follows that $m+1\leq n$. We conclude that $\fm_x^n F=0$. 
\end{enumerate}
This completes the proof of Proposition \ref{Cayley}
\end{proof}

There is another construction possibly giving rise to a slightly different spectral cover of $\Chow_n(\A^d)$. We consider the action of $\fS_{n-1}$ on $(\A^d)^n$  permuting $(x_1,\ldots,x_{n-1})$ and leaving $x_n$ fixed. The categorical quotient $(\A^d)^n \sslash \fS_{n-1}$ is a normal scheme equipped with a morphism $(\A^d)^n \sslash \fS_{n-1} \to (\A^d)^n \sslash \fS_{n}$ which is finite and generically finite étale of degree $n$. We also have a morphism \[\iota: (\A^d)^n \sslash \fS_{n-1}=\Chow_{n-1}(\A^d)\times\A^d \to \Chow_n(\A^d) \times \bA^d\] given by 
$([x_1,\ldots,x_{n-1}],x_n) \mapsto ([x_1,\ldots,x_n],x_n)$.

%$(c'_1,\ldots,c'_{n-1},x_n) \mapsto (c_1,\ldots,c_n,x_n)$ as in \eqref{bullet-equation}. These equations imply that $\iota$ is a closed embedding which identifies $(\A^d)^n \sslash \fS_{n-1}$ with a reduced closed subscheme of $\Chow_n(\A^d)\times \A^d$. 

\begin{proposition}
	 The morphism $\iota:(\A^d)^n \sslash \fS_{n-1} \to \Chow_n(\A^d)\times \A^d$ is a closed embedding. It factors through a universal homeomorphism 
	 \begin{equation} \label{spectral-homeo}
	 	(\A^d)^n \sslash \fS_{n-1} \to \on{Cayley}_n(\A^d)
	 \end{equation}
	which is an isomorphism over $\Chow^\circ_n(\A^d)$.
\end{proposition}

\begin{proof}
We have the following commutative diagram
\[\xymatrix{(\A^d)^n \sslash \fS_{n-1}=\Chow_{n-1}(\A^d) \times \bA^d\ar[d]\ar[r]^{\ \ \ \ \ \ \ \ \ \ \ \ \ \ \iota}&\Chow_n(\A^d)\times \A^d\ar[d]\\
\prod_{i=1}^{n-1} \Sym^i\A^d \times \A^d\ar[r]&\prod_{i=1}^n \Sym^i\A^d \times \A^d}\]
where the vertical arrows are the closed embeddings induced from~\eqref{WeylFFT}
and the lower horizontal arrow is the closed embedding 	
sending 
$(c_1',...,c_{n-1}',x_n)$ to $(c_1,...,c_{n},x_n)$
where $c_1,...,c_n$ are 
given by the equation \eqref{bullet-equation}.
It follows that $\iota$ is a closed embedding.

	%The morphism $\iota:(\A^d)^n \sslash \fS_{n-1} \to \Chow_n(\A^d)\times \A^d$ is obtained by restriction from a morphism $\prod_{i=1}^{n-1} \Sym^i\A^d \times \A^d \to \prod_{i=1}^n \Sym^i\A^d \times \A^d$ given by the equation \eqref{bullet-equation}, which is clearly a closed embedding. It follows that $\iota$ is a closed embedding.
	
	Let $\Chow^\circ_n(\A^d)$ denote the open subscheme of $\Chow_n(\A^d)$ consisting of multiplicity free zero-cycles.  Let us denote $(\bA^d)^{n,\circ}$ the preimage of $B^\circ$ which is the complement in $(\bA^d)^n$ of all diagonals. The morphism $(\bA^d)^{n,\circ}\to \Chow^\circ_n(\A^d)$ is finite, étale and Galois of Galois group $\fS_n$. The morphism $(\bA^d)^{n,\circ}\to (\bA^d)^{n,\circ}\sslash\fS_{n-1}$ is finite, étale, Galois morphism with Galois group $\fS_{n-1}$. It follows that the morphism $(\bA^d)^{n,\circ}\sslash\fS_{n-1}\to \Chow^\circ_n(\A^d)$ is finite, étale of degree $|\fS_n|/|\fS_{n-1}|=n$. 
	
	Over $\Chow^\circ_n(\A^d)$, the morphism $\iota:(\A^d)^{n,\circ} \sslash \fS_{n-1} \to B^\circ\times \A^d$ clearly induces an isomorphism of $(\A^d)^{n,\circ}\sslash \fS_{n-1}$ on $\on{Cayley}_n^\circ(\A^d)$ which is the preimage of $\Chow^\circ_n(\A^d)$ in $\on{Cayley}_n(\A^d)$. Since $(\A^d)^n \sslash \fS_{n-1}$ is an integral scheme, the function $x^n-c_1 x^{n-1}+\cdots+(-1)^n c_n$ which vanishes over $(\A^d)^{n,\circ} \sslash \fS_{n-1}$ has to vanish on all $(\A^d)^n \sslash \fS_{n-1}$. It follows that the morphism $\iota$ factors through a morphism $(\A^d)^n \sslash \fS_{n-1} \to \on{Cayley}_n(\A^d)$. This morphism is finite since $(\A^d)^n \sslash \fS_{n-1}$ is finite over $\Chow_n(\A^d)$. One can check directly that the finite morphism $(\A^d)^n \sslash \fS_{n-1} \to \on{Cayley}_n(\A^d)$ induces a bijection over the $k$-points, which implies that it is a universal homeomorphism.
	\end{proof}

\begin{remark}
Drinfeld asked the question whether the morphism \eqref{spectral-homeo} is an isomorphism, as in the case $d=1$. This is equivalent to saying that $\on{Cayley}_n(\A^d)$ is reduced and normal.
\end{remark}

Recall that in the case $G=\GL_n$, the closed subscheme $B$ of $A$ constructed in Theorem \ref{B} is $B=\Chow_n(\A^d)$. As the universal spectral cover on $B$, we will take 
$$B^\bullet=\on{Cayley}_n(\A^d)$$ 
instead of $(\A^d)^n \sslash \fS_{n-1}$. The reason is that, in Proposition \ref{Cayley}, we have a nice description of the fibers of $B^\bullet$ over $B$, and a generalization of the Cayley-Hamilton theorem.

%The morphism $B^\bullet\to B$ being $\GL_d$-equivariant, it can be twisted by any rank $d$ vector bundle. 
For every geometric point $b\in \cB_X(k)$, we have a morphism $b:X\to [B/\GL_d]$ lying over the morphism $\tau_X^*:X\to \mathbb{B}\GL_d$ corresponding to the cotangent bundle $T_X^*$. By forming the Cartesian product
 \begin{equation}\label{p_b}
	\begin{tikzcd}
X^\bullet_b \arrow{d}{p_b} \arrow{r}
& {[B^\bullet/\GL_d]} \arrow{d} \\
X \arrow{r}{b}
& {[B/\GL_d]}
\end{tikzcd}
\end{equation}
we obtain the spectral cover $X^\bullet_b$ of $X$ corresponding to $b$. Since $B^\bullet \to B$ is a finite morphism, the map $p_b:X^\bullet_B\to X$ is a finite covering. If $b\in \cB_X^\heartsuit$, i.e., $b(X)$ has non-empty intersection with $[B^\circ/\GL_d]$, the covering $p_b:X_b^\bullet \to X$ is generically finite étale of degree $n$.

If $X$ is a curve, and if the spectral curve $X_b^\bullet$ is integral, after Beauville-Narasimhan-Ramanan \cite{BNR}, there is an equivalence of categories between the category of Higgs bundles with spectral datum $b$ and the category of torsion-free $\cO_{X_b^\bullet}$ of generic rank 1. This equivalence can be generalized to the case $d\geq 1$ with the concept of Cohen-Macaulay sheaves. 

Let $M$ be a coherent sheaf on a finite type scheme $Y$.
Let $d=\codim(\Supp(M))$.
Recall that $M$ is called \emph{Cohen-Macaulay} of codimension $d$ if 
$\rH^i(\DD(M))=0$ for $i\neq d$.
A Cohen-Macaulay sheaf $M$ is called \emph{maximal} if it has codimension zero.
%A family of maximal Cohen-Macaulay sheaves on $Y$ parametrized by $S$ is a coherent sheaf $M$ on $Y\times S$  flat over $S$ and satisfying the property: for every $s\in S$, the restriction $M_s$ to $Y\times \{s\}$ is maximal Cohen-Macaulay. If $S$ is Cohen-Macaulay, then the above conditions imply that $M$ itself is maximal Cohen-Macaulay. 
%The functor that associates with every test scheme $S$ the groupoid of all families of maximal Cohen-Macaulay sheaves on $Y$ parametrized by $S$ is a stack \cite{AK2}. 

We also recall an important fact about Cohen-Macaulay modules. 
Suppose that $R$ is a finite $A$-algebra of degree $n$ with $A$ being a regular ring of pure dimension $m$. 
Let $M$ be a $R$-module of finite type. 
Then $M$ is a locally free $A$-module of rank $n$ if and only if 
$M$ is maximal Cohen-Macaulay of generic rank one. 
We refer to \cite[Section 2]{BBG} for a nice discussion on Cohen-Macaulay modules and for further references therein, or the comprehensive treatment in \cite{BH}.

\begin{proposition}
	For every $b\in\cB^\heartsuit_X(k)$, the fiber $h_X^{-1}(b)$ of the Hitchin morphism is isomorphic to the stack of maximal Cohen-Macaulay sheaves of generic rank one on the spectral cover $X^\bullet_b$.
\end{proposition}

\begin{proof}
	Let $(E,\theta)\in h_X^{-1}(b)$ a Higgs bundle of rank $n$ whose spectral datum is $b\in\cB^\heartsuit_X(k)$. Then $E=p_* F$ where $p:T^*X\to X$ is the projection map and $F$ is a coherent sheaf on the cotangent $T^*_X$. By the Cayley-Hamilton theorem, see Proposition \ref{Cayley}, $F$ is supported by the spectral cover $X^\bullet_b \subset T^*_X$. We have then $E=p_{b*} F$ where $p_b:X^\bullet_b\to X$ is the map in ~\eqref{p_b} and $F$ is a coherent sheaf on $X^\bullet_b$. Since $p_b:X^\bullet_b\to X$ is a finite morphism, and $E$ is a vector bundle over $X$, $F$ is a maximal Cohen-Macaulay sheaf. Moreover, since $p_b$ is generically finite étale of degree $n$, $F$ has generic rank one. Conversely, if $F$ is a maximal Cohen-Macaulay sheaf of generic rank one over $X^\bullet_b$, then $E=p_{b*}F$ is a vector bundle of rank $n$ over $X$. It is naturally equipped with a Higgs field $\theta:E \otimes_{\cO_X}\cT_X \to E$ as $X^\bullet_b$ is a closed subscheme of $T^*_X$.
\end{proof}

In spite of the simplicity of the description of $h_X^{-1}(b)$, the proposition above is 
not of great use. For instance, it doesn't imply that $h_X^{-1}(b)$ is non-empty. The difficulty is that in general the spectral cover $X^\bullet_b$  itself might not be Cohen-Macaulay, 
equivalently, the map $X^\bullet_b\to X$ might not be flat, therefore it is not clear how to construct coherent Cohen-Macaulay sheaves on $X^\bullet_b$. At this point, we see that in order to obtain a useful description of $h_X^{-1}(b)$, we need to construct a finite Cohen-Macaulayfication of $X^\bullet_b$. This can be done in the case of surfaces.

\section{Cohen-Macaulay spectral surfaces}

In the case of surfaces, for every $b\in\cB_X^\heartsuit(k)$, the spectral surface $X^\bullet_b$ admits a canonical finite Cohen-Macaulayfication whose construction relies on the theory of Hilbert schemes of points on surfaces and Serre's theorem on extending vector bundles on smooth surfaces across closed subschemes of codimension two. We will first recall Serre's theorem on extending locally free sheaves across a closed subscheme of codimension 2, see \cite[Proposition 7]{Serre}.

\begin{theorem}\label{Serre extension}
	Let $X$ be a smooth surface over $k$, $Z$ a closed subscheme of codimension 2 of $X$ and $j:U\to X$ the open immersion of the complement $U$ of $Z$ in $X$. Then the functor $V\to j_*V$ is an equivalence of categories between the category of locally free sheaves on $U$ and locally free sheaves on $X$. Its inverse is the functor $j^*$. 
\end{theorem}

As we are now considering the case $G=\GL_n$ and $d=2$, the subscheme $B$ of $A=\A^2 \times\Sym^2\A^2 \times \cdots\times\Sym^n \A^2$ is canonically isomorphic to the Chow scheme $\Chow_n(\A^2)$ of zero-cycles of length $n$ on $\A^2$. We recall that a point $b\in \cB_X$ is a section $b:X\to [\Chow_n(\A^2)/\GL_2]$ lying over $\tau_X^*:X\to \mathbb{B}\GL_2$ representing the cotangent bundle $T^*_X$. In other words, $b$ is a section of the relative Chow scheme 
$$\Chow_n(T^*_X/X) \to X$$
obtained from $\Chow_n(\A^2)$ by twisting it by the $\GL_2$-torsor attached to the cotangent bundle $T^*_X$ of $X$. 

Recall the open locus $\Chow_n^\circ(\A^2)$ of $\Chow_n(\A^2)$ consisting of multiplicity free zero-cycles, and $Q$ its complement. Let $\Chow_n^\circ(T^*_X/X)$ the corresponding open locus in $\Chow_n(T^*_X/X)$, and $Q(T^*_X/X)$ its complement. Recall the open locus $\cB^\heartsuit_X$ in $\cB_X$ 
consisting of maps $b: X\to [\Chow_n(\A^2)/\GL_2]$ which maps the generic point of $X$ to the open locus $[\Chow^\circ_n(T^*_X/X)/\GL_2]$. In other words
\begin{equation} \label{A-heart}
	\cB^\heartsuit_X=\{b\in\cB_X| \dim b^{-1}(Q(T^*_X/X))\leq 1\}.
\end{equation}

We first recall some well-known facts about the Hilbert schemes of $0$-dimensional subschemes of a surface, see, for example, \cite{N}. Let $\Hilb_n(\bA^2)$ denote the moduli space of zero-dimensional subschemes of length $n$ of $\bA^2$. A point of $\Hilb_n(\bA^2)$ is a $0$-dimensional subscheme $Z$ of $\bA^2$ of length $n$ that will be of the form $Z=\bigsqcup_{\alpha \in\bA^2} Z_\alpha$ where $Z_\alpha$ is a local $0$-dimensional subscheme of $\bA^2$ whose closed point is $\alpha$. It is known that the Hilbert-Chow morphism 
\begin{equation} \label{Hilbert-Chow}
{\rm HC}_n:\Hilb_n(\bA^2) \to \Chow_n(\bA^2).
\end{equation}
given by $Z\mapsto \sum_{\alpha\in\bA^2} \length(Z_\alpha)\alpha$, where $\length(Z_\alpha)$ is the length of $Z_\alpha$, is a resolution of singularities of $\Chow_n(\A^2)$. It is clear that ${\rm HC}_n$ is an isomorphism over $\Chow^\circ_n(\A^2)$. 

As the morphism \eqref{Hilbert-Chow} is $\GL_2$-equivariant, we can twist it by any $\GL_2$-bundle, and in particular by the $\GL_2$-bundle associated to the cotangent bundle $T^*_X$ over a smooth surface $X$ and by doing so we obtain
\begin{equation}\label{HC-Omega}
	{\rm HC}_{T^*_X/X}:\Hilb_n(T^*_X/X)\to \Chow_n(T^*_X/X).
\end{equation}
This morphism is a proper morphism and its base change to the open subset $\Chow^\circ_n(T^*_X/X)$ is an isomorphism. 
%Here the open immersion $\Chow^\circ_n(T^*_X/X) \subset \Chow_n(T^*_X/X)$ is obtained from the open subscheme $\Chow^\circ_n$ of $\Chow_n$ classifying multiplicity free $0$-cycles by the process of $\GL_2$-twisting by the cotangent bundle. 

\begin{proposition} \label{CM-fication}
	For every $b\in \cB^\heartsuit_X(k)$, there exists a unique finite flat covering 
	\begin{equation}\label{CM surfaces}
		p^{\rm CM}_b:X^{\rm CM}_b\to X
	\end{equation}
	 of degree $n$, equipped with a $X$-morphism $\iota:X^{\rm CM}_b \to T^*_X$ satisfying the following property: there exists an open subset $U\subset X$, whose complement is a closed subset of codimension at least 2, such that $\iota$ is a closed embedding over $U$ and for every $x\in U$, the fiber $(p^{\rm CM}_b)^{-1}(x)$ is a point of $\Hilb_n(T^*_X/X)$ lying over the point $b(x)\in \Chow_n(T^*_X/X)$. Moreover, the morphism $\iota: X^{\rm CM}_b\to T^*_X$ factors through the closed subscheme $X^\bullet_b$ of $T^*_X$ and the resulting morphism $q_b^{\rm CM}:X^{\rm CM}_b\to X^\bullet_b$ is a finite Cohen-Macaulayfication of $X^\bullet_b$.
\end{proposition}

\begin{proof}
	Let $U^\circ$ be the preimage of $\Chow^\circ_n(T^*_X/X)$ by the section $b:X\to \Chow_n(T^*_X/X)$. By assumption $b\in \cB^\heartsuit_X$, $U^\circ$ is a non empty open subset of $X$. As the morphism $\HC_{T^*_X/X}$ of \eqref{HC-Omega} is an isomorphism over $\Chow_n(T^*_X/X)$, we have a unique lifting 
	$$b^{\circ}_{\rm{Hilb}}:U^\circ\to \Hilb_n(T^*_X/X)\times_X U^\circ$$ 
	laying over the restriction $b^\circ=b|U^\circ$. 
	
	Since the Hilbert-Chow morphism \eqref{HC-Omega} is proper, there exists an open subset $U \subset X$, larger than $U^\circ$, whose complement $X-U$ is a closed subscheme of codimension at least 2, such that $b':U^\circ\to \Hilb'_n(T^*_X/X)\times_X U^\circ$ extends to 
	$$b^U_{\rm{Hilb}}:U\to \Hilb_n(T^*_X/X)\times_X U.$$
	By pulling back from $\Hilb_n(T^*_X/X)$ the tautological family of subschemes of $T^*_X$, we get a finite flat morphism $U^+_b \to U$ of degree $n$, equipped with a closed embedding $\iota_U:U^+_b \to T^*_U$.  
	
	According to Serre's theorem on extending vector bundles over surfaces, there exists a unique the finite flat covering $X^{\rm CM}_b\to X$ of degree $n$ extending the finite flat covering $U^+_b$ of $U$. The closed embedding $\iota_U:U^+_b \to T^*_U$ extends to a morphism $\iota:X^{\rm CM}_b \to T^*_X$ which may not be a closed embedding.
	
	By construction $p^{\rm CM}_b:X^{\rm CM}_b \to X$ is a finite flat morphism of degree $n$, it follows from smoothness of $X$ that $X^{\rm CM}_b$ is a Cohen-Macaulay surface. Apply the generalized Cayley-Hamilton theorem to the vector bundle $p^{\rm CM}_{b*}\cO_{X^{\rm CM}_b}$, as $\cO_{T^*_X}$-module over $T^*_X$, it is supported by $X^\bullet_b$. It follows that the morphism $X^{\rm CM}_b\to T^*_X$ factors through a map $q_b^{\rm CM}:X^{\rm CM}_B\to X^\bullet_b\subset T_X^*$. Since $X^{\rm CM}_b$ is finite over $X$, it is also finite over $X^\bullet_b$. As $q_b^{\rm CM}:X^{\rm CM}_b\to X^\bullet_b$ is an isomorphism over the nonempty open subset $U^\circ$, it is a finite Cohen-Macaulayfication of $X^\bullet_b$.
\end{proof}

\begin{remark}\label{torsion free}
Instead of using the Hilbert scheme, we can construct $X^{\rm CM}_b$ over the height one points as follows. Let $U^\circ=b^{-1}(\Chow^\circ_n(\rT^*_X/X))$ and let $Z$ be the complement of $U^\circ$. Let $z$ be the generic point of an irreducible component of $Z$ of dimension one. The localization of $X$ at $z$ is $X_z=\Spec(\cO_{X,z})$ where $\cO_{X,z}$ is a discrete valuation ring. By restricting $p_{b*} \cO_{X^\bullet_b}$ to $\cO_{X,z}$ we get a module of finite type which may have torsion. By considering the quotient $\Spec(p_{b*} \cO_{X^\bullet_b}/(p_{b*} \cO_{X^\bullet_b}^{\rm tors}))$ we obtain a locally free $\cO_{X,z}$-module and thus a section $X_z\to \Hilb_n(\rT^*_X/X)\times_X X_z$ over $b|_{X_z}$. By uniqueness of such a section we have an isomorphism 
\begin{equation}\label{torsion-free-quotient}
	\Spec(p_{b*} \cO_{X^\bullet_b}/(p_{b*} \cO_{X^\bullet_b})^{\rm tors}) \is\Spec(p^{\rm CM}_{b*} \cO_{X^{\rm CM}_b})
\end{equation}
over the complement of a codimension two subscheme of $X$. 
\end{remark}

\begin{remark}
We don't know whether the construction of the 
Cohen-Macaulay spectral surface $X_b^{\rm CM}$ works well in families. 
The issue is that the construction makes use of the 
 equivalence of categories from Theorem \ref{Serre extension} which does not work well in families.
\end{remark}

\begin{theorem}\label{Hitchin fibers}
For every $b\in \cB^\heartsuit_X(k)$, the fiber $h_X^{-1}(b)$ is isomorphic to the stack of Cohen-Macaulay sheaves $F$ of generic rank one over the Cohen-Macaulay spectral surface $X^{\rm CM}_b$. It contains in particular the Picard stack $\cP_b$ of line bundles on $X^{\rm CM}_b$. The action of $\cP_b$ on itself by translation extends to an action of $\cP_b$ on $h_X^{-1}(b)$.

In particular, $h_X^{-1}(b)$ is non-empty.
\end{theorem}

\begin{proof}
Let $(E,\theta)\in\cM_X$ be a Higgs bundle over $X$ lying over $b\in\cB_X^\heartsuit(k)$. The Higgs field $\theta:\cT_{X}\to \End_{\cO_X}(E)$ define a homomorphism $\Sym(\cT_X)\to \End_{\cO_X}(E)$ which factors through $p_{a*} \cO_{X_b^\bullet}$ by the generalized Cayley-Hamilton theorem, see Proposition \ref{Cayley} part 3. 

Let $U^\circ$ and $Z$ be as in Remark \ref{torsion free} and  let $z$ be the generic point of an irreducible component of $Z$ of dimension one. Over $X_z$ we have a homomorphism
$$p_{b*} \cO_{X_b^\bullet}\otimes_{\cO_X} \cO_{X_z} \to \End_{\cO_X}(E)\otimes_{\cO_X} \cO_{X_z}.$$
Since the target is clearly torsion free, this homomorphism factors through \eqref{torsion-free-quotient}. Thus over an open subset $U\subset X$ whose complement is of codimension two, the above morphism factors through a homomorphism of algebras
$$p^{\rm CM}_{b*} \cO_{X^{\rm CM}_b} \otimes_{\cO_X} \cO_U \to \End_{\cO_X}(E)\otimes_{\cO_X} \cO_{U}.$$
By applying Serre's theorem again, we have a canonical homomorphism
$p^{\rm CM}_{b*} \cO_{X^{\rm CM}_b} \to \End_{\cO_X}(E)$. It follows that
$E=\tilde p_{a*} F$ where $F$ is a Cohen-Macaulay $\cO_{X^{\rm CM}_b}$-module of generic rank one. 

Since $p^{\rm CM}_{b}:X^{\rm CM}_b\to X$ is finite flat, for every line bundle $L$ on $X^{\rm CM}_b$, $p^{\rm CM}_{b*} L$ is a vector bundle of rank $n$ carrying a Higgs field. Thus $h_X^{-1}(b)$ contains $\cP_b$. We have an action of $\cP_b$ on $h_X^{-1}(b)$ given by $(L,F) \mapsto L\otimes_{\cO_{X^{\rm CM}_b}} F$ where $L$ is a line bundle on $X^{\rm CM}_b$ and $F$ is a Cohen-Macaulay sheaf of generic rank one.
\end{proof}

%\begin{remark}
%In the case $\dim X=1$, it is known that 
%the Picard stack $\sP_b$ has a dense open orbit in 
%the Hitchin fiber $h_X^{-1}(b)$ when the corresponding spectral curve $X_b^\bullet$ is integral.
%We don't know when $\sP_b$ has an open dense orbit in $h_X^{-1}(b)$ for $\dim X=2$.
%\footnote{According to \cite{AK}, it is an open question when 
%$\Pic(X)$ is dense in the stack of Cohen-Macaulay sheaves on $X$ of generic rank one for a
%higher dimensional variety $X$.}.
%\end{remark}

\begin{remark}
Let $b\in\sB_X^\heartsuit(k)$ such that the Cohen-Macaulay surface $X^{\rm CM}_b$ is integral. 
Consider the functor associating to a $k$-scheme $Y$ the set of isomorphism classes of 
family of 
Cohen-Macaulay sheaves of generic rank one on $X^{\rm CM}_b$ parametrized by $Y$.
According to \cite[Corollary 6.7 and Theorem 7.9]{AK}, the fppf sheafification of this
functor is represented by a $k$-scheme $\Pic(X^{\rm CM}_b)^-$
locally of finite type. 
%which is isomorphic to the Hitchin fiber $h_X^{-1}(b)$
%by Theorem \ref{Hitchin fibers}, 
In addition, $\Pic(X^{\rm CM}_b)^-$ admits a compactification $\Pic(X^{\rm CM}_b)^{=}$ 
whose $k$-points are 
given by isomorphism classes of torsion free rank one sheaves on $X^{\rm CM}_b$. 
 %On the other hand, by \cite[Theorem 6.11]{Simpson 2}, the (extended) Hitchin map $h_X^\square:\cM^\square_X\to\sA_X$ from $\cM^\square_X$, the moduli stack of torsion free Higgs sheaves on $X$, to  $\sA_X$ is proper. One can show that the fiber $(h_X^\square)^{-1}(b)$ is isomorphic to $\Pic(X^{\rm CM}_b)^\square$.

\end{remark}

\begin{definition}
We define $\mathscr B_X^{\Diamond}(k)$ to be the subset of $\mathscr B_X^\heartsuit(k)$
consisting of those points $b$ such that 
the corresponding Cohen-Macaulay spectral surface $X^{\on{CM}}_b$ is 
normal.
\end{definition}

\begin{lemma}
	For $b\in\mathscr B_X^{\Diamond}(k)$,  the neutral component $\cP^0_{b}$ of $\cP_b$ is a quotient of an abelian variety by $\Gm$ acting trivially.
\end{lemma}

\begin{proof}
	This is a consequence of a theorem of Geisser \cite[Theorem 1]{Geisser}. Geisser's theorem states that the multiplicative part of the neutral component $P^0$ of the Picard variety $P$ of an algebraic variety $Y$ is trivial if and only if $\rH^1_{\on{et}}(Y,\Z)$ is trivial whereas the unipotent part is trivial if and only if $Y$ is semi-normal. If $Y$ is normal, $\pi_1(Y)$ is a profinite group, being a quotient of the Galois group of the generic point, and therefore cannot afford a nontrivial continuous homomorphism to $\Z$. It follows that $\rH^1_{\on{et}}(Y,\Z)$ is trivial. On the other hand, a normal variety is certainly also semi-normal. Assume that $X_b^{CM}$ is normal, then the neutral component $P_b^0$ of the Picard variety $P_b$ of $X_b^{CM}$ is an abelian variety. We have $\cP_b^0=[P_b^0/\Gm]$.
\end{proof}

\begin{proposition}
For $b\in\mathscr B_X^{\Diamond}(k)$, 
the action of $\sP_b$ on the Hitchin fiber 
$h_X^{-1}(b)$ is free and $h_X^{-1}(b)$
 is a disjoint union of 
$\sP_b$-orbits.
\end{proposition}
\begin{proof}
If a line bundle $L\in\sP_b$ has a 
stabilizer $F\in h_X^{-1}(b)$ then, as any such $F$, regarding as a sheaf on $X_b^{\on{CM}}$,
is locally free of rank one on the smooth locus $U_b$
of 
$X_b^{\on{CM}}$, the line bundle $L$ is trivial on $U_b$.
Since $X_b^{\on{CM}}$ is normal,
the compliment $X_b^{\on{CM}}\setminus U_b$ is zero dimensional, it implies $L$ is trivial hence the 
action of $\sP_b$ is free.
We claim that 
the 
$\sP_b$ orbits on $h_X^{-1}(b)$ are open and closed.
The closedness follows from the lemma above. 
To show that $\sP_b$-orbits are open, we observe that 
$h_X^{-1}(b)$ is isomorphic to the stack of 
reflexive sheaves of rank one on $X_b^{\on{CM}}$ and, 
for any $F\in h_X^{-1}(b)$, the assignment 
sending $F'\in h_X^{-1}(b)$ to the reflexive hull of 
$F'\otimes_{X_b^{\on{CM}}} F$ (that is, the double dual of 
$F'\otimes_{X_b^{\on{CM}}} F$) defines an 
automorphism of  
$h_X^{-1}(b)$ mapping $\sP_b$ isomorphically to the 
$\sP_b$-orbit through $F$.
Since $\sP_b$ is open in $h_X^{-1}(b)$ (see \cite{AK}), it implies 
that $\sP_b$-orbits are open in $h_X^{-1}(b)$. The proposition follows.

\end{proof}

We expect that $\mathscr B_X^{\Diamond}(k)$ is a non-empty open subset of $\mathscr B_X(k)$
for most algebraic surfaces. The non-emptiness of $\mathscr B_X^{\Diamond}(k)$
is closely related to 
questions on
zero locus of symmetric differentials, which seems 
very little is known in higher dimension.

%%%%%%%%%%%%%%%%%%%%
\quash{
\begin{conjecture}\label{non-emptiness}
$\mathscr B_X^{\Diamond}(k)$ is a non-empty open subset of $\mathscr B_X(k)$.
\end{conjecture}

The conjecture above is closely related to 
question on
zero locus of symmetric differentials. 
%which very little is known in the higher dimensional case.
For example,
for $G=\GL_2$, one can identify
$\mathscr B_X(k)$ with the subspace of 
$\mathscr A_X(k)=\rH^0(X,S^1\Omega_X^1)\oplus\rH^0(X,S^2\Omega_X^1)$ consisting of
pairs
$(a_1,a_2)$
such that $a_1^2-4a_2\in\rH^0(X,S^2\Omega^1_X)$ is locally of the form 
$u^2$ for some $u\in\rH^0(X,S^1\Omega_X^1)$, and 
the subspace $\mathscr B^\Diamond_X(k)\subset \mathscr B_X(k)$
consists of those
pairs $(a_1,a_2)$ with the property that 
the zero locus of the quadratic differential $a_1^2-4a_2$ is either zero dimensional or 
a multiplicity free divisor. In the case of fibered surfaces,
one can use the theory of  
Hitchin morphism for curves to show that  generic points in $\mathscr B_X(k)$ satisfy the property above, hence confirms Conjecture \ref{non-emptiness} in this case (see Section \ref{fibered surfaces}). For general surfaces, e.g., surfaces of general type, it appears that very little is known about 
zero locus of quadratic differentials.
}
 %%%%%%%%%%%%%%%%

\section{Surfaces fibered over a curve}\label{fibered surfaces}

In this section we investigate the spectral surfaces $X^\bullet_b$ and the Cohen-Macaulay spectral surface $X^{\rm CM}_b$ in the case when $X$ is a fibration over a 
curve $C$ and apply our findings to ruled and elliptic surfaces. 

Let $X$ be a proper smooth surface and let 
$C$ be a proper smooth curve. Assume there is a proper flat 
surjective map $\pi:X\to C$ such that the generic fiber is a proper smooth curve.
We denote by $X^0\subset X$ the largest open subset such that $\pi$ is smooth.
Consider the cotangent morphism $d\pi:T_C^*\times_CX\to T_X^*$.
It induces a map \[[d\pi]:\on{Chow}_n(T_C^*/C)\times_CX
%\is\on{Chow}_n(T_C^*\times_CX/X)
\to\on{Chow}_n(T_X^*/X)\]
on the relative Chow varieties. 
For every section 
$b_C:C\to\on{Chow}_n(T_C^*/C)$, the composition 
$$b_X:X\stackrel{}\is C\times_CX\stackrel{a_C\times\on{id}_X}\to\on{Chow}_n(T_C^*/C)\times_CX\stackrel{[d\pi]}\to\on{Chow}_n(T_X^*/X)$$
is a section of $\on{Chow}_n(T_X^*/X)\to X$ and the assignment $b_C\to b_X$
defines a map
\beq\label{embedding}
\iota_\pi:\mathscr B_C\to\mathscr B_X.
\eeq
We claim that the map above is a closed embedding. To see this we observe that 
there is a commutative diagram 
\beq\label{maps between bases}
\xymatrix{\mathscr B_C\ar[r]^{\iota_\pi}\ar[d]^{}&\mathscr B_X\ar[d]^{}\\
\mathscr A_C\ar[r]^{j_\pi}&\mathscr A_X}
\eeq
where the vertical arrows are the natural embeddings, and the bottom arrow is the embedding 
\[j_\pi:\mathscr A_C=\bigoplus_{i=1}^n\rH^0(C,\Sym^i\Omega_C^{1})\hookrightarrow\mathscr A_{X}=\bigoplus_{i=1}^n\rH^0(X,\Sym^i\Omega^1_X)\]
induced by the injection of vector spaces 
$\rH^0(C,\Sym^i\Omega_C^1)= \rH^0(X,\pi^*\Sym^i\Omega_C^1)\to\rH^0(X,\Sym^i\Omega_X^1)$. 
The claim follows. Note that, since $\on{dim} C=1$, the left vertical arrow in (\ref{maps between bases}) is in fact an isomorphism. From now on we will view $\mathscr B_C$ as a subspace of $\mathscr B_X$. Since the cotangent map 
$d\pi:T_C^*\times_CX\to T_X^*$ is a closed imbedding over the 
open locus $X^0$, we have $$\mathscr B_C^\heartsuit=\mathscr B_C\cap\mathscr B_X^\heartsuit.$$ 

For any $b\in\mathscr B_C$, we denote by $C^\bullet_b\to C$ the corresponding spectral curve and we define $X^+_b=C^\bullet\times_C X$. 
The natural projection map $p^+_b:X^+_b\to X$ is finite flat of degree $n$. Since $X$ is smooth, it follows that $X^+_b$ is a Cohen-Macaulay surface.

\begin{lemma}\label{Y_a}
There exits a finite $X$-morphism
$q^+_b:X^+_b\to X^\bullet_b$ which is a generic isomorphism
if $b\in\mathscr B_C^\heartsuit$. If the fibration $\pi:X\to C$ has only reduced fibers, then 
for any $b\in\mathscr B_C^\heartsuit$, the map 
$q^+_b:X^+_b\to X^\bullet_b$ is isomorphic to the 
finite Cohen-Macaulayfication
$q^{\rm CM}_b: X^{\rm CM}_b\to X^\bullet_b$ in Proposition \ref{CM-fication} (which is well-defined since $b\in\mathscr B_X^\heartsuit$).
\end{lemma}

\begin{proof}
Let $i^+_b:X^+_b\to T^*_X$ be the restriction of 
the cotangent morphism $d\pi:T_C^*\times_CX\to T^*_X$
to the closed sub-scheme $X^+_b\subset T_C^*\times_CX$. By the 
Cayley-Hamilton theorem the map $i^+_b$ factors through 
the spectral surface $X^\bullet_b$. Let $q^+_b:X^+_b\to X^\bullet_b$ be the 
resulting map. 
As $X^+_b$ is finite over $X$, the map $q^+_b$ is finite.
In addition, if $b\in\mathscr B_C^\heartsuit$,
then both $X^+_b$ and $X^\bullet_b$ are generically \'etale over 
$X$ of degree $n$ and it implies that $q^+_b$ is a generic isomorphism.
 
Assume the fibers of $\pi$ are reduced. 
Then the smooth locus $X^0$ of the map $\pi$ is open and its complement 
$X-X^0$ is a closed subset of codimension $2$.
Since the map $i^+_b:X^+_b\to T^*_X$
is a closed embedding over $X^0$, 
Proposition \ref{CM-fication} implies 
the finite flat covering    $q^+_b:X^+_b\to X^\bullet_b$ 
is isomorphic to the finite Cohen-Macaulayfication $q^{\rm CM}_b: X^{\rm CM}_b\to X^\bullet_b$. 
\end{proof}

\begin{definition}
We define $\mathscr B_C^{\Diamond}$ to be the open subset of $\mathscr B_C^\heartsuit$
consisting of those points $b$ such that 
the corresponding spectral curve $C_b$ is 
smooth and 
irreducible. 
\end{definition}

\begin{corollary}\label{normality}
Assume the fibration $\pi:X\to C$ has only reduced fibers. Then 
we have $\mathscr B_C^{\Diamond}\subset\mathscr B_X^{\Diamond}$, that is,
the surface $X^{\on{CM}}_b$ is normal for $b\in\mathscr B_C^\Diamond$.
\end{corollary}

\begin{proof}
Since $X^{\on{CM}}_b$ is Cohen-Macaulay, by Serre's criterion for normality,
it suffices to show that the $X^{\on{CM}}_b\is X^+_b$ is smooth in codimension $\leq 1$. 
The assumption implies the complement $X-X^0$ has codimension at least 2.
Since $C_b$ is smooth for $b\in\mathscr B_C^\Diamond$,
 the open subset $X^{+0}_b:=\tC_a\times_C X^0\subset X^+_b$ is 
smooth (since the map $X^{+0}_b\to C_b$ and $C_b$ are smooth) and 
the complement $X^+_b-X^{+0}_b$ has codimension at least 2. The corollary follows.
\end{proof}

\begin{example}
Consider the case when $X=C\times\mathbb P^1$
and $n=2$. We have 
$\sB_X=\sB_C=\rH^0(C,\Omega_C^1)\oplus\rH^0(C,\Sym^2\Omega_C^1)$.
Let $b=(b_1,b_2)\in\sB^\heartsuit_C$ and $p_{b}:X^\bullet_b\to X$ be the corresponding
spectral surface. Then \'etale locally over $X$, 
the surface $X^\bullet_b$ is isomorphic to the closed subscheme of 
 $\on{Spec}(k[x_1,x_2,t_1,t_2])$ defined by the equations 
\begin{equation}\label{phe2}
\left\{\begin{split} t_1^2+b_1t_1+b_2=0\\ t_2(2t_1+b_1)=0
 \\\ t_2^2=0
\end{split}
\right.
\end{equation}
here $x_1,x_2$ are local coordinate of $C$ and $\mathbb P_1$
and $b_i\in k[x_1]$.
Let $\on{discr}_C=(b_1^2-4b_2=0)\subset C$ be the discriminant divisor for $b$.  
From (\ref{phe2}) we see that  $X^\bullet_b$ is an \'etale 
cover of degree $2$ away from the divisor  
$\on{discr}_C\times\mathbb P^1\subset X$. Note  
that the spectral surface $p_{b}:X^\bullet_b\to X$ is not flat over $X$ as the push-forward
$p_{b*}\mO_{X^\bullet_b}$ has length three over 
$\on{discr}_C\times\mathbb P^1$. 
The finite Cohen-Macaulayfication $X^{\on{CM}}_b\to X^\bullet_b$ is 
given by the flat quotient 
$\on{Spec}(p_{b*}\mO_{X^\bullet_b}/(p_{b*}\mO_{X^\bullet_b})^{\on{tors}})$
which is isomorphic to $X^{\on{CM}}_b\is C_b\times\mathbb P^1$. 
The Hitchin fiber $h_X^{-1}(b)$ is isomorphic to 
$$h_X^{-1}(b)=h_C^{-1}(b)\times\sP ic(\mathbb P^1)= h_C^{-1}(b) \times B\Gm.$$
\end{example}

\begin{prop} \label{1-form}
	Let $X$ be a smooth projective surface and $\pi:X\to C$ be either a ruled surface, or a non-isotrivial elliptic surface with reduced fibers. Then for every $n$, the pull-back map $$\rH^0(C,\Sym^n \Omega^1_C) \to \rH^0(X,\Sym^n \Omega^1_X)$$ 
	is an isomorphism. 
	\end{prop}
	
It follows from the proposition above that in the case of ruled surfaces and non-isotrivial elliptic surface with reduced fibers, we have $\cA_C=\cA_X$. Since $\cB_C=\cA_C$, we have $\cB_X=\cB_C$
and $\cB^\Diamond_X$ and $\cB^\heartsuit_X$ are open dense in $\cB_X$.
 For every $b\in \cB_C$, we have a spectral curve $C^\bullet_b$ which is finite flat of degree $n$ over $C$. We also have the spectral surface $X^\bullet_b$ which is a finite scheme over $X$ embedded in its cotangent bundle $T_X^*$. The Cohen-Macaufication of $X^\bullet_b$ is $X^+_b=C_b\times_C X$. In the case of elliptic surfaces, the morphism $X^{\rm CM}_b\to X^\bullet_b$ may not be an isomorphism, and $X^{\rm CM}_b$ may not be embedded in the cotangent bundle $T^*_X$. The existence of the Cohen-Macaulay spectral cover guarantees that $h_X^{-1}(b)$ is non-empty. 

The Proposition \ref{1-form} is obvious for ruled surfaces. Let us investigate it in the case of elliptic surfaces. We assume there is a proper flat map 
$\pi:X\to C$ from $X$ to a smooth projective curve $C$ with general fiber 
a smooth curve of genus one. We will focus on the case when $\pi:X\to C$  
is not isotrivial, relatively minimal, and has reduced fibers (e.g., semi-stable non-isotrivial elliptic surfaces). Let $X^0$ denote the largest open subset of $X$ such that the restriction of $\pi$ to $X^0$ is a smooth morphism $\pi^0 :X^0\to C$. Since the geometric fibers of $\pi$ are all reduced, the complement of $X^0$ in $X$ is a zero-dimensional subscheme. 
Over $X^0$, we have an exact sequence of tangent bundles 
\begin{equation} \label{ses-tangent}
	0\to\cT_{X^0/C}\to\cT_{X^0}\to(\pi^0)^*\cT_{C}\to 0
\end{equation}
For every $n\in\N$, we have the exact sequence of symmetric powers 
\begin{equation}\label{ses-sym}
	0 \to \Sym^{n-1} \cT_{X^0} \otimes \cT_{X^0/C} 
	\to \Sym^n \cT_{X^0} \to (\pi^0)^*\Sym^n  \cT_{C}\to 0.
\end{equation}

Let $\eta\in C$ be the generic point of $C$ and let $X_\eta=X\times_C\eta$ 
which is
an elliptic curve over $\eta$. The restriction of \eqref{ses-tangent} to $X_\eta$ is a short exact sequence making the rank two vector bundle $\cT_X|_{X_\eta}$ a self-extension of the trivial line bundle of $X_\eta$. As we assume the elliptic fibration $\pi$ is non isotrivial, i.e., the Kodaira-Spencer map is not zero, $\cT_X|_{X_\eta}$ is a non-trivial self-extension of the trivial line bundle on $X_\eta$. After Atiyah \cite{A}, such a  non-trivial extension is unique up to isomorphism
\begin{equation} \label{Atiyah}
	0 \to \cO_{X_\eta} \to \cE \to \cO_{X_\eta} \to 0.
\end{equation}
In other words, the restriction of \eqref{ses-tangent} to the generic fiber $X_\eta$ is isomorphic to \eqref{Atiyah}. 

\begin{lemma}
	The exact sequence of symmetric powers derived from \eqref{Atiyah}
\begin{equation} \label{Sym-Atiyah}
	0 \to \Sym^{n-1} \cE \to \Sym^n \cE \to \cO_{X_\eta} \to 0
\end{equation}
is not split.
\end{lemma}

\begin{proof}
	 Indeed if 
\begin{equation} \label{ses}
	0\to\cL' \to\cE \to\cL \to 0
\end{equation}
is an extension of of a line bundle $\cL$ by a line bundle $\cL'$, then there is a canonical filtration 
$$0=\cF_0 \subset \cF_1 \subset \cdots \subset \cF_{n-1} \subset \cF_n =\Sym^n\cE$$ 
of $\Sym^n\cE$ such that for every $i\in\{1,\ldots,n\}$ we have 
$\cF_{i}\is\Sym^i\cE\otimes\cL'^{\otimes n-i}$ and
$\cF_{i}/\cF_{i-1} \is\cL^{\otimes i} \otimes {\cL'}^{\otimes n-i}$. Moreover, the exact sequence 
\begin{equation} \label{ses-shift}
	0 \to \cF_{n-1}/\cF_{n-2} \to \cF_n/\cF_{n-2}\to \cF_{n}/\cF_{n-1} \to 0
\end{equation}
is isomorphic to the sequence \eqref{ses} tensored by $\cL^{\otimes(n-1)}$. In particular, if \eqref{ses} is not split, then \eqref{ses-shift} is not split either, and as a consequence, the exact sequence
$$0\to \cF_{n-1} \to \cF_n \to \cL^{\otimes n} \to 0$$
is not split.
Applying above discussion to~\eqref{Atiyah}, we see that~\eqref{Sym-Atiyah} is not split.

\end{proof}

\begin{lemma} For every $n\in\N$, we have 
	\begin{enumerate}
		\item $\dim_K \Ext^1(\cO_{X_\eta}, \Sym^n \cE)=1$
		\item $\dim_K(\Hom(\Sym^n \cE,\cO_{X_\eta}))=1$
		\item The restriction map $\Hom(\Sym^n \cE,\cO_{X_\eta})\to \Hom (\Sym^{n-1} \cE,\cO_{X_\eta})$ is zero.  
	\end{enumerate}
\end{lemma}

\begin{proof}
	It follows from induction on $n$ using the Ext long exact sequences derived from \eqref{Sym-Atiyah}.
\end{proof}

It follows from the above lemmas that, for every $n\in\N$,	
$\Sym^n \cE$ is the unique extension of $\cO_{X_\eta}$ by $\Sym^{n-1}\cE$, up to isomorphism.

Now we prove that pulling back 1-forms defines an isomorphism 
$$\rH^0(C,\Sym^n \Omega^1_{C}) \is \rH^0(X,\Sym^n \Omega^1_{X}).$$
This map is obviously injective, let us prove that it is also surjective.
A symmetric form $\alpha \in \rH^0(X,\Sym^n \Omega^1_{X})$ gives rise to a linear form $\alpha: \Sym^n \cT_X\to \cO_X$. By restriction to the generic fiber $X_\eta$ of the elliptic fibration, we obtain a map $\alpha_\eta:\Sym^n \cE\to \cO_{X_\eta}$. By previous lemma, the restriction of $\alpha_\eta$ to $\Sym^{n-1}\cE$ is zero. It follows that in the exact sequence \eqref{ses-sym}, the restriction of $\alpha$ to $\Sym^{n-1} \cT_{X^0} \otimes \cT_{X^0/C}$ is zero, i.e., it factors through $(\pi^0)^*\Sym^n\cT_{C}$. Since the complement of $X^0$ in $X$ is zero dimensional, $\alpha$ factors through $(\pi)^*\Sym^n\cT_{C}$, i.e., it comes from a symmetric form on $C$. 
This finishes the proof of Proposition \ref{1-form}.

These calculations show that the Hitchin morphism for ruled and elliptic surfaces are closely related to the Hitchin morphism for the base curve. This is compatible with the fact that under the Simpson correspondence
\cite{Simpson 1}, stable Higgs bundles for a smooth projective surface $X$ correspond to irreducible representations of the fundamental group $\pi_1(X)$, and in the case of ruled surfaces and non-isotrivial elliptic surfaces with reduced fibers, we have $\pi_1(X)\is\pi_1(C)$ where $C$ is the base curve (see, e.g., \cite[Section 7]{Friedman}).

\newpage
\section*{Acknowledgement}

\foreignlanguage{vietnamese}{Ngô Bảo Châu}'s research is partially supported by NSF grant DMS-1702380 and the Simons foundation. He is grateful \foreignlanguage{vietnamese}{Phùng Hồ Hải} for stimulating discussions in an earlier stage of this project. He also thanks Gérard Laumon for many conversations on the Hitchin fibrations over the years and his encouragement. The research of Tsao-Hsien Chen is partially supported by NSF grant DMS-1702337.
He thanks Victor Ginzburg and Tomas Nevins for useful discussions.
We thank Vladimir Drinfeld for useful comments on an earlier draft of this paper. 
We thank the anonymous referees for their valuable comments.

\end{document}